\pdfoutput=1

\documentclass[reqno]{amsart} 
\usepackage[T1]{fontenc}
\usepackage{hyperref}
\usepackage{color}
\usepackage{ulem}
\usepackage{soul}

\newtheorem{Ex}{Example}[section]

\usepackage{comment}

\usepackage[numbers]{natbib}

\makeatletter
\def\subsection{\@startsection{subsection}{2}%
  \z@{.5\linespacing\@plus.7\linespacing}{.1\linespacing}%
  {\normalfont\normalsize\bfseries}}
\makeatother

\usepackage{amsmath}
\usepackage{amsfonts,amssymb,amsthm,mathrsfs}
\usepackage{bbm}
\usepackage{scalerel}

\usepackage{multirow}

\newcommand{\weight}{w}
\newcommand{\veps}{\varepsilon}
\newcommand{\tot}{\mathrm{tot}}

\newcommand{\G}{\mathrm{G}}

\newcommand{\R}{\mathbb{R}}
\newcommand{\N}{\mathbb{N}}
\newcommand{\C}{\mathscr{C}}

\newcommand{\D}{\mathcal{D}}
\newcommand{\probClass}{\mathcal{P}(a,b)}
\newcommand{\probClassi}{\mathcal{P}(a_i,b_i)}
\newcommand{\funClass}{\C^{0,1}_{+}(a,b)}

\newcommand{\funClassTwo}{\C^{1,2}_{*}(a,b)}
\newcommand{\funClassTwoi}{\C^{1,2}_{*}(a_i,b_i)}
\newcommand{\funClassWeight}{\mathcal{W}(a,b)}
\newcommand{\funClassWeighti}{\mathcal{W}(a_i,b_i)}
\newcommand{\muref}{\Tilde{\mu}}
\newcommand{\weightId}{\weight_{\text{lin}}}

\newcommand{\weightG}{\weight_{\text{G}}}

\newcommand{\weightIdi}{\weight_{i,\text{lin}}}
\newcommand{\weightGi}{\weight_{i,\text{G}}}

\newcommand{\weightUi}{\weight_{i,\text{U}}}

\newcommand{\weightSTIMi}{\widehat{\weight_{i}}}
\newcommand{\DGSMi}{\nu_{i,\weight_i}}

\newcommand{\gId}{g_{\text{lin}}}

\newcommand{\fSTIMi}{\widehat{f_i}}

\newcommand{\nn}{{(n)}}
\newcommand{\kk}{{(k)}}
\newcommand{\br}{\,|\,}
\newcommand{\brr}{\,\bigr|\,}

\def\abs#1{\left| #1 \right|}

\def\Var{\mathop{\rm Var}\nolimits}

\newcommand{\Esp}{\mathbb{E}}
\def\Rnot#1{{\fontfamily{qcr}\selectfont #1}}

\usepackage{xcolor}

\usepackage{bm} 
\usepackage{subcaption} 
\usepackage{float}
\usepackage{ragged2e} 

\usepackage{xspace} 
\newcommand{\PCE}{PoinCE\xspace}

\newtheorem{theo}{Theorem}[section]

\newtheorem{Def}{Definition}[section]
\newtheorem{prop}{Proposition}[section]

\numberwithin{equation}{section}

\newcommand\blfootnote[1]{%
  \begingroup
  \renewcommand\thefootnote{}\footnote{#1}%
  \addtocounter{footnote}{-1}%
  \endgroup
}

\newcommand{\mathsubsection}[2]{
  \subsection*{#1 \texorpdfstring{$#2$}{}}
  \addtocontents{toc}{\protect\setcounter{tocdepth}{1}}
  \leavevmode
}

\title[Weighted Poincar\'e inequalities for Global Sensitivity Analysis]{On one dimensional weighted Poincar\'e inequalities for Global Sensitivity Analysis}
\author{David Heredia, Ald\'eric Joulin, Olivier Roustant}
\date{}

\begin{document}
\maketitle

\begin{abstract}
One-dimensional Poincar\'e inequalities are used in Global Sensitivity Analysis (GSA) to provide derivative-based upper bounds and approximations of Sobol indices. We add new perspectives by investigating weighted Poincar\'e inequalities. Our contributions are twofold.
In a first part, we provide new theoretical results for weighted Poincar\'e inequalities, guided by GSA needs.
We revisit the construction of weights from monotonic functions, providing a new proof from a spectral point of view. In this approach, given a monotonic function $g$, the weight is built such that $g$ is the first non-trivial eigenfunction of a convenient diffusion operator. This allows us to reconsider the ``linear standard'', \textit{i.e.} the weight associated to a linear $g$. In particular, we construct weights that guarantee the existence of an orthonormal basis of eigenfunctions, leading to approximation of Sobol indices with Parseval formulas. In a second part, we develop specific methods for GSA. We study the equality case of the upper bound of a total Sobol index, and link the sharpness of the inequality to the proximity of the main effect to the eigenfunction. This leads us to theoretically investigate the construction of data-driven weights from estimators of the main effects when they are monotonic, another extension of the linear standard. Finally, we illustrate the benefits of using weights on a GSA study of two toy models and a real flooding application, involving the Poincar\'e constant and/or the whole eigenbasis.
\end{abstract}

\tableofcontents

\section{Introduction}

\blfootnote{\textit{2020 Mathematics Subject Classification.} 26D10, 39B62, 37A30, 60J60, 62G05, 62P30, 65C60.  

\textit{Key words and phrases.} Weighted Poincar\'e inequality, spectral gap, Global Sensitivity Analysis, Sobol-Hoeffding decomposition, Sobol indices, weighted Derivative-based Global Sensitivity Measures, Poincar\'e chaos expansion.}

\subsection{Motivation} 
The development of Global Sensitivity Analysis (GSA) of numerical model outputs has become increasingly popular in the three last decades. It is by now an essential international research topic, combining modern mathematical and statistical tools to computer experiments, and has many consequences in engineering for industry. Recall that the principle of GSA is to quantify the influence of input random variables on the output of a multivariate function $f: \R^d \to \R$ which might be expensive to evaluate since dimension $d$ is large. These variables can represent calculation codes that model complex phenomena or artificial intelligence algorithms whose functioning is not well understood.

When it comes to quantifying influence or uncertainty, practitioners tend to use Sobol indices \cite{sobol1993,sobol2001} because of their clear interpretation in terms of ANOVA decomposition, at least under the assumption of independent entries. 
More precisely, they are defined according to the Sobol Hoeffding decomposition
\begin{equation*}
    f(X) = \sum_{I\subset \{1,\dots,d\}}f_{I}(X_I),
\end{equation*}
where $X = (X_1,\ldots, X_d)$ is the $d$-dimensional random vector of independent inputs $X_i$ and provided the output random variable $f(X)$ lies in $L^2$. Above each $X_I$ is the random vector formed by the variables $X_i$ with $i\in I$ and the terms $f_I(X_I)$ are uniquely characterized by the non-overlapping property
\begin{equation*}
    \Esp[f_I(X_I) \br X_J]=0, \quad \mbox{for all }J\subsetneq I,
\end{equation*}
where by convention $\Esp[\, \cdot \, |X_J] = \Esp[\, \cdot \, ]$ when $J=\emptyset$. Such conditions imply the orthogonal decomposition of the total variance
\[
\Var(f(X))=\sum_{I\subset \{1,\dots,d\}} \Var(f_I(X_I)).
\]
In particular, the main effects $f_i(X_i)$ carry the influence of each variable individually and the total contributions are captured by the total effects 
$$
f_i^\tot(X) = \sum_{I\ni i}f_I (X_I).
$$ 
The latter naturally defines total Sobol indices as the percentage of variance explained by them:
\begin{equation*}
    S_i^\tot = \frac{\Var(f_i^\tot(X))}{\Var(f(X))} \, \in \, [0,1].
\end{equation*}
Despite their clear interpretability,
the estimation of total Sobol indices requires numerous calculations, making them an expensive computational tool. 
When the derivatives of $f$ are available, other sensitivity indices called DGSM (Derivative-based Global Sensitivity Measure) reveal to be efficient since they are cheaper to compute, cf. \cite{SobolKucherenko2009, SobolKucherenko2010}. As observed in \cite{lamboni} and further studied in detail in \cite{poincareintervals}, Sobol indices
and DGSM are connected by a one-dimensional Poincar\'e inequality. In other words, when it is satisfied, such a functional inequality provides an upper bound of the variance-based index by using the derivative-based one, easier to handle in practice. Hence DGSM indices can be seen as a credible alternative to Sobol indices for screening purposes, allowing to identify input variables with minimal influence when the devoted DGSM indices are sufficiently small (balanced with the other components appearing in the upper bound). To apply these techniques, providing some information on the Poincar\'e constant, \textit{i.e.}, the best constant in the Poincar\'e inequality, 
is of crucial importance. Since in theory it is quite hard to find explicitly those objects even in our one-dimensional context, some numerical methods are required. As presented in \cite{poincareintervals}, they are mainly based on a combination of the spectral interpretation of the Poincar\'e constant together with an appropriate finite element discretization which is relevant in the context of small dimension.

Actually, there is absolutely no reason to limit ourselves to bound from above Sobol indices only by DGSM ones, at least for two reasons. On the one hand there exists some usual probability measures which do not satisfy the usual Poincar\'e inequality (\textit{e.g.}, heavy-tailed distributions) and on the other hand the DGSM-based upper bounds might be too large. To overcome this difficulty, we shall use another functional inequality involving the variance and alternative quantities still constructed with respect to the derivatives of the function $f$. Hence we are naturally led to introduce appropriate weight functions $w_i$ in the DGSM indices. It gives rise to the notion of weighted DGSM indices
\begin{equation*}
    \DGSMi = \Esp\left[\weight_i(X_i)\left( \frac{\partial f}{\partial x_i}(X) \right)^2\right] ,
\end{equation*}
provided the expectation makes sense, the univariate dependence of the chosen weights being natural within the present one-dimensional context. Such functional inequalities are called weighted Poincar\'e inequalities and yield to the key upper bound
\begin{equation*}
    S_i^\tot\leq C_P(\mu_i,\weight_i) \, \frac{\DGSMi}{\Var (f(X))} ,
\end{equation*}
where $\mu_i$ stands for the distribution of the input random variable $X_i$ and the Poincar\'e constant $C_P(\mu_i,\weight_i)$ depends on the weight $w_i$ (the unweighted case corresponding to the classical choice $\weight_i\equiv 1$). Hence it provides an additional degree of freedom by choosing conveniently the weight to enhance the precision of the upper bound. In particular it suggests, among other things, the construction of data-driven weights in order to improve the classical (unweighted) results offered in \cite{poincareintervals} when applied to specific models arising in the GSA methodology.

To conclude with motivations, notice that this framework is not limited to provide only upper bounds on Sobol indices. Indeed, under slightly additional assumptions, we may consider the spectral interpretation related to the weighted Poincar\'e inequalities, giving rise to the so-called Poincar\'e chaos, somewhat similar to that emphasized in \cite{Chaos2,PoincareChaos}. See also \cite{Adcock2007} for another usage of spectral expansions through Sturm-Liouville operators. Then, all Sobol indices can be expressed as Parseval identities and truncations of these identities give relevant derivative-based lower bounds on those indices. Here also, considering a weight in Poincar\'e inequalities provides an additional degree of freedom that should help to improve these lower bounds.

To illustrate the strength of weighted Poincar\'e inequalities in GSA, we provide in Figure \ref{fig: intro} some numerical computations of total Sobol indices and associated estimated bounds on a GSA standard model (we refer to Section \ref{Applications} for more details). We can see that, compared to the results of the unweighted case reported in \cite{poincareintervals}, using weights clearly improves the accuracy of the upper and lower bounds.
\begin{figure}[ht]
\begin{subfigure}{.5\textwidth}
  \centering
  \includegraphics[width=1\linewidth]{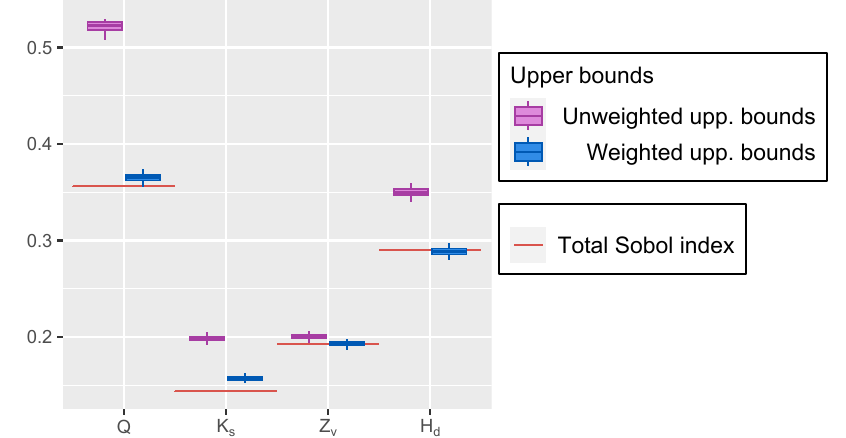}
\end{subfigure}%
\begin{subfigure}{.5\textwidth}
  \centering
  \includegraphics[width=1\linewidth]{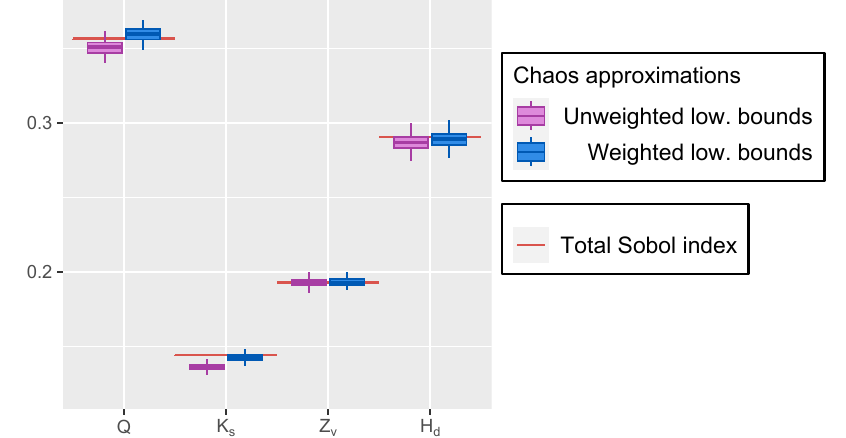}
\end{subfigure}
\captionsetup{margin={0cm, 0cm},justification=justified,singlelinecheck=false}
  \caption{
  Illustration of the benefits of using weights in Poincar\'e inequalities for bounding total Sobol indices on the hydrological problem of Section \ref{Application to a flood model}.}
\label{fig: intro}
\end{figure}

\subsection{Contributions and plan}
The present work is divided into two main parts: first we provide new theoretical results for  weighted Poincar\'e inequalities, driven by GSA needs, whereas in a second time we develop specific applications to GSA. Finally we illustrate numerically our results on toy models to observe the relevance of our approach and in particular on a flood model. Here is below a detailed description of our main contributions together with a plan of the paper. 

Section 2 offers an introduction to theoretical aspects of weighted Poincar\'e inequalities and their corresponding spectral interpretation. Then our main contributions on these weighted functional inequalities are contained in Section \ref{sect:main}. In a nutshell, they can be summarized as follows. \\ 
Firstly, we revisit the construction of weights presented in \cite{germain_swan}, offering a new proof from a spectral point of view and without requiring to the formalism of Stein's method. In addition, we provide a necessary and sufficient condition for obtaining non-vanishing weights. In this construction, given some suitable monotonic function $g_i$, we determine the weight $w_i$ for which $g_i$ is the first non-trivial eigenfunction of a convenient diffusion operator involving $w_i$ as diffusion constant. As a result, $g_i$ is essentially the only function to saturate the weighted Poincar\'e inequality. Furthermore, we provide a numerical method to approximate the weight $\weight_i$ when it does not admit a closed-form expression, a situation encountered in GSA when the distribution of the input variables are non-standard. \\ 
Secondly, we further investigate the common weight choice associated to the case where $g_i$ is linear. This case was studied for GSA in \cite{Song} by means of a calculus of variations approach (Euler-Lagrange equations). We expand their list of weights by including several examples related to truncated or heavy-tailed probability measures. Beyond this standard, we propose a new way to construct non-vanishing weights, based on a reference probability measure. \\ 
Thirdly, we consider the case where the weighted Poincar\'e inequality is not saturated, \textit{i.e.} when the first eigenfunction does not exist. Thanks to the intertwining approach proposed in \cite{bj,bjm}, we are able to provide the exact value of the Poincar\'e constant for original and new examples involving a large class of weights. \\ 
In Section \ref{sect:GSA_link}, we develop GSA-oriented results and in particular data-driven weights. We study the equality case of the upper bound of a total Sobol index, linking the sharpness of the inequality to the proximity of the main effect $f_i$ with respect to the eigenfunction $g_i$. This justifies the common weight choice associated to a linear $g_i$ for models that exhibit almost linear main effects. But this also suggests the construction of data-driven weights from estimators of the main effects when they are monotonic (and not only linear). In some sense, this is a data-driven version of the general weight proposed in Section \ref{sect:main}. We prove the consistency of this weight estimator and the associated upper bound. \\ 
Our final Section \ref{sect:Appli} is dedicated to applications. We collect all the previous material to study two toy models and a real flooding application, involving the Poincar\'e constant and/or the whole eigenbasis. In particular our numerical results on the flood model reveal to be relevant and exhibit a serious improvement of the ones established in \cite{poincareintervals} through the usual (\textit{i.e.} unweighted) Poincar\'e inequalities and Poincar\'e chaos.  

To conclude this introduction, we point out that our approach is limited to one-dimensional functional inequalities for the moment but it would be very natural to generalize our results to higher dimensions. A nice first result was presented
in the recent article \cite{cui2024optimal} in which the authors generalize the Stein approach to some multi-dimensional log-concave probability measures such as moment measures of convex functions. In some sense this work can be seen as a multi-dimensional extension of \cite{germain_swan} by means of Stein kernel weights, and also of \cite{Song} although the methods emphasized are somewhat different. Nevertheless this generalization to higher dimension remains open in full generality and would have many challenging and interesting consequences from a GSA perspective since then high-dimensional independent inputs $X_i$ with non necessarily independent coordinates could be addressed. We hope that such a direction will be the matter of future research. 

\section{Background on weighted Poincar\'e inequalities}
\label{Spectral interpretation}
We start by briefly introducing some theoretical aspects about one-dimensional weighted Poincar\'e inequalities, which are the main protagonists of the present paper.

Given $-\infty\leq a<b \leq \infty$, we denote by $\funClass$ the set of functions that are continuous and piecewise $\C^1$ (\textit{i.e.} continuously differentiable) on $[a,b]$ and positive on $(a,b)$. We define $\funClassTwo$ as the class of 
functions $f$ such that $f'\in \funClass$ or $-f'\in \funClass$. We denote $\probClass$ the set of probability measures $\mu$ on $[a,b]$ with density (with respect to the Lebesgue measure) $\rho\in \funClass$ satisfying $\rho (a) >0$ (resp. $\rho (b) >0$) when $a$ (resp. $b$) is finite.
Let $\funClassWeight$ be the set of continuous functions on $[a,b]$, that are piecewise $\C^1$ and positive on $(a,b)$.
In the sequel we systematically refer to functions $w\in \funClassWeight$ as weights.
Note that $\funClassWeight$ differs from $\funClass$ since we do not require weights to be differentiable at the boundaries.
In the case where $a$ and/or $b$ are infinite, we adopt the convention $[-\infty,b]=(-\infty,b]$ and/or $[a,\infty]=[a,\infty)$, to avoid making the distinction every time. 
It is of course possible to consider a more general setting, but this one is relevant and fulfills our purposes.

Let $L^2(\mu)$ be the space of square-integrable functions with respect to some probability measure $\mu \in \probClass$ and, given some weight function $w\in \funClassWeight$, denote $H^1(\mu,\weight)$ the weighted Sobolev space defined as
\[
H^1(\mu,\weight)=\left\{f\in L^2(\mu)\brr \weight^{1/2} f'\in L^2(\mu)\right\},\]
where $f'$ stands for the weak derivative of the function $f$. We are now in position to give the definition of a weighted Poincar\'e inequality.
\begin{Def}
A probability measure $\mu \in \probClass$ satisfies a weighted Poincar\'e inequality with weight function $\weight \in \funClassWeight$ and constant $C>0$ if for every function $f\in H^1(\mu,\weight)$ such that $\int_a^b f \, d\mu =0$ (we say that $f$ is centered), it is true that
    \begin{equation}
    \label{eq: poincar\'e}
    \scaleobj{1.0}{
        \int_{a}^b f^2 \, d\mu 
\leq C\int_a^b \weight \, (f')^2 \, d\mu .}
\end{equation}
The Poincar\'e constant, denoted $C_P(\mu,\weight)$, is the optimal (\textit{i.e.}, the smallest) constant $C$ for which (\ref{eq: poincar\'e}) holds. If there exists some non-null function which realizes the equality in (\ref{eq: poincar\'e}), we say that it saturates the weighted Poincar\'e inequality.
\end{Def}
In the particular case where $\weight\equiv 1$, we are reduced to the classical Poincar\'e inequality which has been largely studied in the literature. See for instance \cite{BGL} for an introduction to the topic, with precise references and credit.

Actually, such an idea to consider weighted functional inequalities of Poincar\'e type takes roots at least in the 70s within the pioneer work of Brascamp and Lieb \cite{brascamp_lieb} in a multi-dimensional log-concave context. Moreover it reveals to have strong consequences in high-dimensional analysis, in connection with other important functional inequalities, isoperimetry and concentration of measure, cf. for instance the work of Bobkov and Ledoux \cite{bob_ledoux} about general convex measures including heavy-tailed distributions, \textit{i.e.}, probability distributions whose tails decay is slower than exponential. In the one-dimensional case, the analysis can be further explored either by considering Hardy-type inequalities and Sturm-Liouville equations \cite{ bobkov_gotze1, bobkov_gotze2,muckenhoupt} or using Stein's method as in the papers \cite{germain_swan,saumard} in which the weight corresponds to the so-called Stein kernel. Recently, new theoretical guarantees have been proposed in \cite{bj,bjm} by means of the intertwining technique. Such an approach will be developed in Section \ref{sect:intert}.  

Similarly to the usual one $C_P(\mu,1)$, the Poincar\'e constant $C_P(\mu,\weight)$ admits a dual interpretation related to a convenient diffusion operator. To observe this, we need to introduce a bit of structure. Denote $\C^\infty(a,b)$ the space of infinitely differentiable real-valued functions on $[a,b]$ and consider the subspace
$$
\C_{N,w}^\infty(a,b)=\left\{f\in \C^\infty(a,b)\br \weight(a)f'(a)\rho(a) =\weight(b)f'(b)\rho(b) =0\right\} ,
$$
the presence of the index $N$ standing for Neumann boundary conditions adapted to the presence of the weight $w$  (these are the only boundary conditions we will consider throughout this paper). If $a$ and/or $b$ are infinite, the boundary conditions above are understood as taking the limit $a \to -\infty$ and/or $b\to \infty$. Then the diffusion operator of interest is defined on $\C_{N,w}^\infty(a,b)$ 
as follows: 
\begin{equation*}
    \label{eq: diffusion operator}
    L_w f=\frac{1}{\rho}\left(\weight f' \rho\right)' = w f'' + \left( w' + w (\log \rho )' \right) \, f' . 
\end{equation*} Note that this operator involves the weight $w$ as diffusion constant, in contrast to the canonical operator usually related to the probability measure $\mu$, \textit{i.e.}, the one constructed with the choice $w \equiv 1$. Trivial integrations by parts tell us that the operator $-L_w$ is symmetric on $\C_{N,w}^\infty(a,b) \subset L^2 (\mu)$ and non-negative, \textit{i.e.}, 
for every $f , g \in \C^\infty_{N,w} (a,b)$,
\begin{equation*}
\label{eq: symmetry}
\int_a^b (-L_w f) \, g\, d\mu = \int_a^b \weight \, f' \, g' \, d\mu = \int_a^b f \, (-L_w g) \, d\mu , 
\end{equation*}
and 
\begin{equation}
\label{eq: positivity}
\int_a^b (-L_w f) \, f\, d\mu = \int_a^b \weight \, (f' )^2 \, d\mu \geq 0,  
\end{equation}
respectively. If $[a,b]$ is finite and 
$w$ does not vanish at the boundary, the operator is said to be regular according to the formalism of Sturm-Liouville problems and it is essentially self-adjoint in $L^2 (\mu)$, \textit{i.e.}, it admits a unique self-adjoint extension (still denoted $-L_w$) with domain $\D ( -L_w) \subset L^2 (\mu)$ in which the space $\C^\infty_{N,w} (a,b)$ is dense for the operator norm. Otherwise the operator is called singular and to ensure this essentially self-adjointness property, the metric induced by the operator (or rather by the so-called \textit{carr\'e du champ} operator, thus by $w$) is assumed to be complete. We refer to \cite{Zettl} for a classical reference about Sturm-Liouville problems and to Chapter 3 of \cite{BGL} for the essentially self-adjointness property studied in the context of general diffusion Markov triples.

Once the (unique) self-adjoint extension is defined, let us introduce some elements about the eigenvalues and eigenfunctions of the operator $-L_w$. Roughly speaking, the eigenvalues correspond to a part of the spectrum $\sigma(-L_w) \subset \R^+$ given by the values $\lambda$ for which there exists a centered function $g \in \D (-L_w)$ satisfying $-L_w g=\lambda g$, \textit{i.e.}, an eigenfunction associated to $\lambda$. Note that our definition of eigenvalue differs a bit from the usual definition since $\lambda$ may not be isolated. In the present context, the first eigenvalue is $\lambda_0(-L_w) = 0$ and the associated one-dimensional eigenspace is generated by the constant eigenfunction $e_0\equiv 1$. If $0$ is isolated in the spectrum, it means that the second element of the spectrum defined by the variational formula 
\begin{equation}
\label{eq: Courant}
\lambda_1(-L_w)=\inf_{\substack{f\in H^1(\mu,\weight) \\ f \scalebox{.7}{\mbox{ centered}} }}\frac{\int_a^b (-L_w f) \, f \, d\mu}{\int_{a}^b f ^2 \, d\mu},
\end{equation}
is positive. The quantity $\lambda_1(-L_w) \, (= \lambda_1(-L_w)-\lambda_0(-L_w) )$ is called the spectral gap of the diffusion operator $-L_w$. Note that it may not be an eigenvalue since the infimum above is not always reached. Using \eqref{eq: positivity}, we deduce that the Poincar\'e constant $C_P (\mu,w)$ admits the following spectral interpretation: 
$$
C_P(\mu,\weight) = \frac{1}{\lambda_1(-L_w)} .
$$ 
In this way, finding the Poincar\'e constant $C_P (\mu,w)$ is equivalent to identify the spectral gap $\lambda_1(-L_w)$, a dual interpretation which will be systematically used in our paper. Moreover, if the eigenspace related to the spectral gap is non empty, it is also one-dimensional and a given function $e_1$ (say) is an associated eigenfunction if and only if it saturates the weighted Poincar\'e inequality \eqref{eq: poincar\'e}. In particular it satisfies for all $g\in H^1(\mu,\weight)$,
\begin{equation}
\label{eq: variational formulation}
\int_a^b g\, e_1 \, d\mu = C_P(\mu,\weight) \, \int_a^b g \, (-L_w e_1 ) \, d\mu = C_P(\mu,\weight) \, \int_a^b \weight\, g' \, e'_1 \, d\mu .
\end{equation}

To finish this introduction about weighted Poincar\'e inequalities, we point out that it is a difficult task in general to give the explicit value of the spectral gap, except maybe for some particular examples of probability measures and weights. However there exists a useful method in the present one-dimensional setting which allows to identify it when the eigenfunction $e_1$ exists. By \cite{Chen} we know that $e_1$ is the only eigenfunction of the operator $-L_w$ such that its derivative does not vanish on $(a,b)$, is of constant sign and satisfies $\weight(a)f'(a)\rho(a) = \weight(b)f'(b)\rho(b)=0$. In other words, if we find some eigenfunction satisfying these properties, then the associated eigenvalue is necessarily the spectral gap. Actually, this observation is the main idea behind our main results to which we turn now.

\section{Main results on weighted Poincar\'e inequalities}
\label{sect:main}
\label{Result for optimal weights}
\subsection{A general result on the construction of weights}

In this part, we revisit the formulation of the optimal weight built by means of the Stein method in \cite[Theorem 3.5]{ernst_reinert_swan2020} (with $l=0$) and in \cite[Theorem 2.4]{germain_swan}, offering a new proof from a spectral point of view. In other words, Theorem \ref{teo: weight} below determines the weight for which a suitably selected function saturates the weighted Poincar\'e inequality. 
In particular this result extends the weights provided in \cite{Song} with linear saturating functions. \\
Notice that each weight and corresponding Poincar\'e constant are uniquely defined up to a multiplicative constant since $C_P(\mu, k w) = k^{-1} C_P(\mu, w)$ for every $k>0$. In this section we adopt the normalization $C_P(\mu,\weight)=1$.

\begin{theo}
\label{teo: weight}
Let $\mu \in \probClass$ and
let $g$ be a centered function in $\funClassTwo$. 
If $a$ (resp. $b$) is finite we assume that $g'(a)\neq 0$, or that $g'(a)=0$ and $g''(a)\neq 0$ (resp. $g'(b)\neq 0$, or that $g'(b)=0$ and $g''(b)\neq 0$). 
Then
\begin{itemize}
\item The function
\begin{equation}
\label{eq: optimalweight}
\weight_g(x) = -\frac{1}{g'(x)\rho(x)}\int_{a}^x g(y)\rho(y)\,dy,\quad x\in (a,b), 
\end{equation}
belongs to $\funClassWeight$. If $a$ is finite,
the value of $w_g(a)$   depends on $g$:
\begin{itemize}
\item If $g'(a)\neq 0$, then $\weight_g(a)=0$.
\item If $g'(a)=0$ and $g''(a)\neq 0$, then $\weight_g(a)=-g(a)/g''(a)>0$.
\end{itemize}
The same conclusion holds for $w_g(b)$ if $b$ is finite.
\item  If in addition $g\in L^2(\mu)$ (which is satisfied if $a$ and $b$ are finite), then the  weighted Poincar\'e inequality \eqref{eq: poincar\'e} holds with weight $\weight_g$ and Poincar\'e constant $C_P(\mu,\weight_g) = 1$. Furthermore, the inequality is saturated by $g$.
\end{itemize}
\end{theo}

\begin{proof}
We start dealing with the regularity of our weight. The definition of $\weight_g$ in \eqref{eq: optimalweight} provides its continuity on $(a,b)$. We multiply by $g'\rho$ in both sides and differentiate to obtain 
\begin{equation}
\label{eq: teo optimal weight 1}
(\weight_g g'\rho)' =
-g \rho ,
\end{equation}
so that we get on $(a,b)$, 
\begin{equation*}
\weight_g' = - \frac{g}{g'} - \weight_g \left(\frac{g''}{g'}+\frac{\rho'}{\rho}\right).
\end{equation*}
Then, the regularity of $g$ and $\rho$ together with the fact that $g'$ and $\rho$ do not vanish on $(a,b)$ show that $\weight_g'$ is piecewise continuous and thus $\weight_g$ is piecewise $\C^1$ on $(a,b)$. 

Let us prove now that $w_g$ is positive on $(a,b)$. The function $g$ is centered and increasing, hence $\lim_{x\to a} g(x) < 0$, $\lim_{x\to b} g(x) >0$ and there exists a unique $c\in (a,b)$ such that $g(c) = 0$. Denoting $G(x) := \int_a^x g(y)\rho(y)\,dy$, $x\in (a,b)$, we have the following table:  

\begin{table}[ht]
\begin{tabular}{clcccr}
                        & $a$ &          & $c$ &          & $b$         \\ \hline 
\multicolumn{1}{c|}{$G'$} & \multicolumn{2}{c}{-} &  $0$  & \multicolumn{2}{c}{+}
\\ 
\multicolumn{1}{c|}{$G$}  & \multicolumn{2}{c}{\hspace{0.3 cm}\rotatebox[origin=c]{-20}{$\scaleobj{1.4}{\longrightarrow}$}}  & & \multicolumn{2}{c}{\rotatebox[origin=c]{+20}{$\scaleobj{1.4}{\longrightarrow}$} \hspace{0.3 cm}}
\end{tabular}
\end{table}
\noindent and since $\lim_{x\to a} G(x) = \lim_{x\to b} G(x) =0$, we obtain $G <0$ and thus $w_g >0$ on $(a,b)$. 

Next, if $a$ is finite the value $w_g(a)$ is defined as the limit of $w_g(x)$ when $x\rightarrow a$, guaranteeing its right-continuity at point $a$
(the value $w_g(b)$ is obtained identically when $b$ is finite, providing the left-continuity at $b$). Namely if $g'(a)\neq 0$, then we observe immediately from the definition of $\weight_g$ that $\weight_g(a)=0$. Otherwise if $g'(a)=0$ and $g''(a)\neq 0$, then we rewrite $w_g$ as 
$$
\weight_g (x) = -\frac{x-a}{(g'(x)-g'(a)) \, \rho(x)} \times \frac{1}{x-a} \, \int_{a}^x g(y)\rho(y)\,dy,\quad x\in (a,b), 
$$
so that taking the limit $x\to a$ entails that  $\weight_g(a)=-g(a)/g''(a)$.

Finally, to prove that the weighted Poincar\'e inequality \eqref{eq: poincar\'e} is satisfied with weight $\weight_g$ and Poincar\'e constant $C_P(\mu,\weight_g) = 1$ when $g\in L^2 (\mu)$, we consider the spectral interpretation of the Poincar\'e constant. Equality (\ref{eq: teo optimal weight 1}) indicates that $g$ is an eigenfunction of (minus) the operator 
$$
L _{w_g} f = \frac{(\weight_g f' \rho)'}{\rho} , $$ with associated eigenvalue $\lambda = 1$. Furthermore, since the derivative of $g$ does not vanish on $(a,b)$, then the spectral gap $\lambda_1(-L_{w_g})$ coincides necessarily with this eigenvalue, meaning that 
$$ 
C_P(\mu,\weight_g) = \frac{1}{\lambda_1(-L_{w_g})} = 1,
$$ 
and that the weighted Poincar\'e inequality is saturated by $g$. The proof of Theorem \ref{teo: weight} is now complete.
\end{proof}
Notice that the assumptions on $g$ in Theorem \ref{teo: weight} mimick the properties required to be an eigenfunction $e_1$ associated to the inverse Poincar\'e constant. Indeed, first we know that $e_1$ is centered and is the only eigenfunction (up to a multiplicative constant) of the operator $-L _{w_g}$ satisfying $e_1'>0$ (up to a change of sign). Secondly, if $[a,b]$ is a finite interval, the possible boundary conditions are consistent with the values of $\weight_g$ at the boundary 
(we have $\weight_g(a)g'(a)\rho(a)=\weight_g(b)g'(b)\rho(b)=0$) 
and prevent it from exploding.

As we will see in the sequel, Theorem \ref{teo: weight} is one of the result on which our forthcoming numerical study is based and thus will be used many times in our paper, in particular when dealing with applications to GSA. When the weight $w_g$ provided by \eqref{eq: optimalweight} does not admit a closed-form expression, we will implement a numerical method that approximates it for any suitable pair of probability measure $\mu$ and function $g$. This method is further detailed in Section \ref{Numerical computation}. Before that, let us revisit the classical situation where $g$ is a linear function and investigate new relevant weights provided by Theorem \ref{teo: weight}.

\subsection{Revisiting the optimal weight for linear saturating functions}
\label{sect:linear} 
The classical weight $\weightId$ (say) used in the literature corresponds in Theorem \ref{teo: weight} to the linear choice 
$$
g(x) = \gId(x) = x - \int_a^b y\,\rho (y) \, dy, 
$$ 
which is often convenient since linear functions saturate the devoted weighted Poincar\'e inequality. This weight is given by the formula
$$
\weightId (x) = - \frac{1}{\rho(x)} \, \int_{a}^x \gId (y) \, \rho(y)\,dy,\quad x\in [a,b]. 
$$
In particular we will see in Section \ref{Case of equality in the upper bound and stability} that this choice is optimal when approximating linear phenomena in the GSA context. Notice that $\weightId$ has to vanish at the boundary (if non empty) since $\gId' \equiv 1$ does not vanish at the boundary. Following this strategy, the authors in \cite{Song} give a list of closed-form expressions for $\weightId$ associated to classical laws, including the uniform, exponential and Gaussian distributions. 

However it is necessary to consider other examples frequently encountered in GSA problems, typically truncated distributions. Heavy-tailed measures are also considered, for which the presence of a weight becomes necessary to establish a Poincar\'e-type inequality (recall that they do not satisfy the classical one due to the lack of exponential integrability, see for instance \cite{BGL}). Table \ref{table: weights} below provides some of those examples, including two heavy-tailed distributions: the generalized Cauchy measure $\mu_\beta$ of parameter $\beta >1/2$, denoted $\mathcal{C}(\beta)$, whose density is defined as  
$$
\rho (x) = \frac{1}{Z_\beta \, (1+x^2)^{\beta}}, \quad x\in \R ,  
$$ 
with $Z_\beta$ the normalization constant, and the Pareto distribution $\mu_{z,\alpha}$ with parameters $z>0$ and $\alpha>0$, denoted $\mathcal{P} ar (z,\alpha)$, with density given by 
$$
\rho (x) = \frac{\alpha z^\alpha} {x^{\alpha+1}} , \quad x \geq z.
$$
According to Theorem \ref{teo: weight}, both measures satisfy a weighted Poincar\'e inequality with their respective weights $\weightId$ if they admit a finite second moment. This condition is fulfilled when $\beta>3/2$ and $\alpha>2$, respectively. Table \ref{table: weights} also includes the exact expression of the weight $\weightId$ for some truncated probability measures. We refer to Appendix \ref{Computations of weights adapted for linear functions} for detailed computations.
All the weights vanish at the boundary (when non empty), as expected according to Theorem \ref{teo: weight}. Moreover they converge pointwise as $h \to \infty$ to the weight $\weightId$ on the whole space (provided it is well-defined, \textit{i.e.} $\gId$ is square-integrable), a property which is also true in full generality. Regarding the truncated versions of $\mathcal{C}(\beta)$ and $\mathcal{P}ar (z,\alpha)$, we point out that the parameters $\beta$ and $\alpha$ can take any real value since the Lebesgue density is always continuous on the truncated interval and thus no extra integrability condition is required.

\begin{table}[ht]
\begin{center}
\Large
\resizebox{\textwidth}{!}{
\begin{tabular}{|c|cc|}
\hline
Probability measure   $\mu$              & \multicolumn{2}{c|}{Weight $\weightId$}   \\  \hline 
\hline 

\multirow{2}{*}{Uniform $\mathcal{U}(a,b)$}    & \multicolumn{2}{c|}{\multirow{2}{*}{$\frac{1}{2}(x-a)(b-x)$}} \\
                             & \multicolumn{2}{c|}{}                  \\ \hline
\multirow{2}{*}{
\vspace{-0.5 cm}
Exponential $\mathcal{E}(\gamma)$} & \multicolumn{1}{c|}{
On $\R^+$}  & Truncated on $[0,h]$ \\ \cline{2-3} 
                             & \multicolumn{1}{c|}{$\displaystyle\frac{x}{\gamma }$}       &     $\displaystyle\frac{1}{\gamma}\left(x-h \, \frac{1-e^{\gamma x}}{1-e^{\gamma h}}\right)$    \\ \hline
\multirow{2}{*}{\vspace{-0.5 cm}
Normal $\mathcal{N}(m,\sigma^2)$}      & \multicolumn{1}{c|}{On $\R$}  & Truncated on $[m-h,m+h]$ \\ \cline{2-3} 
                             & \multicolumn{1}{c|}{$\displaystyle \sigma^2$}       &       $\displaystyle \sigma^2\left(1-\exp\left(\frac{(x-m)^2}{2\sigma^2}-\frac{h^2}{2\sigma^2}\right)\right)$   \\ \hline
\multirow{2}{*}{
$\begin{array}{c}
       \\ \\ \\
    \mbox{Gen. Cauchy } 
    \mathcal{C}(\beta) \\
     \\
\end{array}
$
}      & \multicolumn{1}{c|}{On $\R$}  & Truncated on $[-h,h]$ \\ \cline{2-3} 
                       &
                       \multicolumn{1}{c|}{
                       $
                       \begin{array}{c}
                           \mbox{For }\beta>3/2:\hspace{0.2 cm} \\ 
                           \vspace{-0.3 cm} 
                           \\    
     \displaystyle\frac{1+x^2}{2(\beta-1)} 
                       \end{array}
                       $
                       
                       }       &
                       $\begin{array}{c}
                         \mbox{For }\beta\neq 1: \hspace{8.2 cm}   \\  \displaystyle\frac{1}{2(\beta-1)}\left((1+x^2)-(1+x^2)^\beta (1+h^2)^{-\beta+1}\right)
                         \\
                        \mbox{For }\beta=1: \hspace{8.2 cm} \\     \displaystyle \frac{1}{2}(1+x^2)\log\left(\frac{1+h^2}{1+x^2}\right)
                       \end{array}$   \\ \hline
                       
\multirow{2}{*}{
$\begin{array}{c}
     \\ \\  \\
    \mbox{Pareto } \mathcal{P} ar (z,\alpha) \\
     \\
     \\
\end{array}
$
}      & \multicolumn{1}{c|}{On $[z,\infty)$}  & Truncated on $[z,z+h]$ \\ \cline{2-3} 
                       &
                       \multicolumn{1}{c|}{
                       $
                       \begin{array}{c}
                       \mbox{For }\alpha>2:
                       \hspace{0.7 cm}\\
                       \vspace{-0.3 cm} 
                       \\
                          \displaystyle \frac{x(x-z)}{\alpha-1}
                       \end{array}$
                                              }       &
                       $
                       \begin{array}{c}
                        \mbox{For }\alpha\neq 1:\hspace{8.2 cm}\\ 
                       \displaystyle \frac{x^{\alpha+1}}{\alpha-1} \left(\frac{z^{1-\alpha}-(z+h)^{1-\alpha}}{z^{-\alpha}-(z+h)^{-\alpha}}(z^{-\alpha}-x^{-\alpha})-(z^{1-\alpha}-x^{1-\alpha})\right) \\
                        \displaystyle 
                        \mbox{For }\alpha=1: \hspace{8.2 cm}\\   
                        \displaystyle x^2\left(\log\left(\frac{z+h}{z}\right)\frac{z^{-1}-x^{-1}}{z^{-1}-(z+h)^{-1}}-\log\left(\frac{x}{z}\right)\right)
                       \end{array}
                       $ 
                       \\ \hline
                       
\end{tabular}
}
\end{center}
\caption{Examples of weight $\weightId$.}
\label{table: weights}
\end{table}

\subsection{Beyond $\weightId$: non-vanishing weights on finite intervals}
\label{sect:non_vanishing}
Theorem \ref{teo: weight} allows the generation of weights with saturating functions beyond linear ones. Furthermore, it enables to consider non-vanishing weights. This case is particularly relevant in GSA since then the existence of an orthonormal basis of eigenfunctions is guaranteed, at least when the interval $[a,b]$ is finite (the associated Sturm-Liouville problem is regular, cf. \cite{Zettl}).

According to Theorem \ref{teo: weight}, adding the requirement that the desired weight does not vanish forces the associated saturating function $g$ to satisfy very specific conditions. Namely, in addition to be a centered function in $\funClassTwo$, $g$ must verify
\begin{equation}
    \label{eq: boundary conditions}
    g'(a)=g'(b)=0, \quad g''(a)\neq 0\quad \mbox{and}\quad g''(b)\neq 0.
\end{equation}
A simple way to build such a suitable function is to consider another ``reference'' probability measure $\muref$ and to look at a function $\Tilde{g}$ saturating the corresponding classical Poincar\'e inequality. Indeed, $\Tilde{g}$ is an eigenfunction for the diffusion operator related to $\muref$ with (non-vanishing) constant weight equal 1, so that $\Tilde{g} \in \funClassTwo$ and satisfies the Neumann conditions $\Tilde{g}'(a)=\Tilde{g}'(b)=0$. It is easy to deduce that $\Tilde{g}$ then also verifies $\Tilde{g}''(a)=\Tilde{g}''(b)=0$. The function $g$ is then obtained by centering $\tilde{g}$ with respect to $\mu$. We collect all these elements in the following proposition. 
\begin{prop} 
\label{teo:refMeasure}
Consider two probability measures $\mu, \Tilde{\mu} \in \probClass$ with $(a,b)$ finite. Let $\Tilde{g}$ be the function saturating the classical Poincar\'e inequality for $\Tilde{\mu}$. Then the function $g$ defined by $g (x) = \Tilde{g} (x) - \int_{a}^{b} \Tilde{g}(y) \,\mu(dy)$ belongs to $\funClassTwo$ and satisfies \eqref{eq: boundary conditions}. Moreover it generates a non-vanishing weight $w_g\in \funClassWeight$ leading to a weighted Poincar\'e inequality for $\mu$. \end{prop}

There are several ways to define $\Tilde{\mu}$.
First, as an extension of $\weightId$, we can define $\muref$ on the finite interval $[a,b]$ for which $\tilde{g}$ is close to be linear. In the context of GSA, the associated weight will be adapted to linear phenomena (as it is the case with $\weightId$), as explained by the stability argument given in Section \ref{Case of equality in the upper bound and stability}. Since the function $\gId$ saturates the classical Poincar\'e inequality for the normal distribution on the real line, 
the idea is to choose $\muref$ as a truncated Gaussian measure highly concentrated in $[a,b]$. For instance, we consider for $\muref$ the distribution $\mathcal{N}\left(\frac{a+b}{2},\sigma^2\right)\bigr|_{[a,b]}$, with variance $\sigma^2$ chosen such that $\mathbb{P}\left(\mathcal{N}\left(\frac{a+b}{2},\sigma^2\right)\in [a,b]\right)=0.95$. Then the corresponding function $\Tilde{g}$ is the so-called Kummer function given in terms of an hypergeometric series, see for instance \cite{poincareintervals}. 
As expected, numerical computations suggest that $\Tilde{g}$ is approximately linear except near the boundary since its derivatives vanish at these points. In the sequel, we denote by $\weightG$ the weight $w_g$ associated to this choice of $\muref$ (by Proposition \ref{teo:refMeasure}) in order to emphasize the role of the Gaussian distribution.

Secondly, another idea is to choose a convenient reference probability measure $\Tilde{\mu}$ that allows simple computations. A natural candidate is the uniform distribution $\mathcal{U}(a,b)$. It is well-known that the optimal constant in the classical Poincar\'e inequality is $C_P(\Tilde{\mu},1) = (b-a)^2/\pi^2$ and the corresponding saturating function is  $\Tilde{g} (x) = \cos\left(\pi(x-a)/(b-a)\right)$. We note $\weight_{\mathrm{U}}$ the generated weight.

\subsection{Numerical computation}
\label{Numerical computation}
In this short part we provide a numerical method to approximate the weight $\weight_g$ from any eligible probability measure $\mu \in \probClass$, with $[a,b]$ finite, and function $g$. The idea is to solve the Cauchy problem
     \begin{equation*}
     \label{eq: RK4}
       \left\{\begin{array}{rl}
        (\weight_g g'\rho)'(x)&=-g(x)\rho(x) \quad \mbox{on }(a,b),\\
        (\weight_g g'\rho)(a)&=0,
       \end{array}\right.
     \end{equation*}
and then simply divide the solution by $g'\rho$. Above $\rho$ stands for the Lebesgue density of $\mu$. The method is implemented using the R software. To solve the differential equation we use the Runge-Kutta 4 method (in R, the function \Rnot{rk4} from the package \Rnot{deSolve} \cite{R_deSolve}) which is very accurate. Indeed, it is known that if $h$ denotes the size of the uniform partition of the interval $[a,b]$, the approximation error at any point is of order $O(h^4)$, cf. for instance \cite{BurdenNumerical}. Note that the estimated weight vanishes at points $a$ and $b$. To prevent inconsistencies when $\weight_g(a) \neq 0$ or/and $\weight_g(b)\neq 0$, we apply smooth corrections near the boundaries. We illustrate the precision of the numerical method by computing weights for the standard uniform distribution and a truncated normal one. Figures \ref{fig: Uniform} and \ref{fig: Normal} below display the theoretical weights $\weightId$, together with their numerical approximations and the numerical approximations of $\weight_\mathrm{U}$ and $\weightG$. To compute $\weightG$, we initially approximate the function that generates it. This is carried out through the finite element method as in \cite{poincareintervals}. Note that for the case of the uniform distribution, $\weight_U$ takes the constant value $C_P(\mu,1)=1/\pi^2$ (so that the probability measure $\mu$ and the reference one $\Tilde{\mu}$ are the same).
     \begin{figure}[ht]
         \centering
         \includegraphics[scale=0.5]{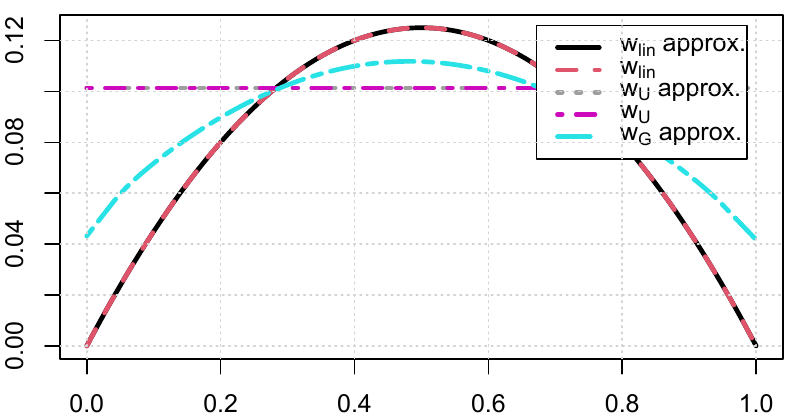}
         \caption{The weights $\weightId(x)=\frac{1}{2}x(1-x)$, $\weight_\mathrm{U} (x)=1/\pi^2$, their numerical approximations and the numerical approximation of $w_\G$, associated to the uniform distribution $\mathcal{U}(0,1)$.}
         \label{fig: Uniform}
     \end{figure}
     \begin{figure}[ht]
         \centering
\includegraphics[scale=0.5]{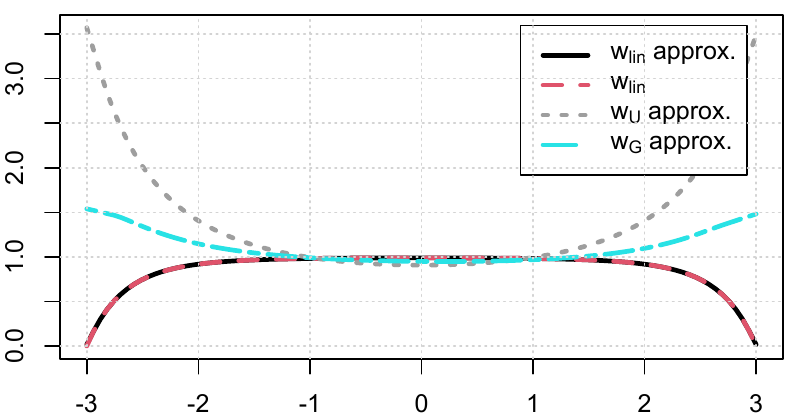}
         \caption{The weight $\weightId(x)=\displaystyle 1-\exp\left(x^2/2-9/2\right)$, its numerical approximation and the numerical approximations of $w_\mathrm{U}$ and $w_G$, associated to the truncated normal distribution $\mathcal{N}(0,1)$ on the interval $[0,3]$.}
         \label{fig: Normal}
     \end{figure}

Finally, Figure \ref{fig: Gumbel} below shows the numerical approximations of $\weightId$, $\weightG$ and $\weight_\mathrm{U}$ associated to a truncated Gumbel distribution that naturally appears in hydrology according to extreme value theory (its Lebesgue density is given in Section \ref{Application to a flood model}). We do not have a close expression for any of the weights in this case.
          \begin{figure}
         \centering
         \includegraphics[scale=0.5]{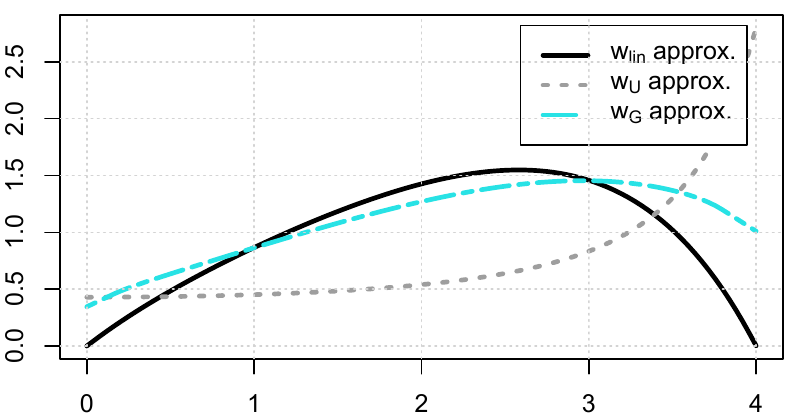}
         \caption{Numerical approximations of $\weightId$, $w_\mathrm{U}$ and $w_G$ associated to the truncated Gumbel distribution $\mathcal{G}(0,1)$ truncated on $[0,4]$.}
         \label{fig: Gumbel}
     \end{figure}
     
\subsection{When the weighted Poincar\'e inequality is not saturated}
\label{sect:intert}
As we have seen previously in Theorem \ref{teo: weight}, our strategy is to construct a weight such that the functions saturating the corresponding weighted Poincar\'e inequality can be identified. In other words, those functions are the eigenfunctions with associated eigenvalue the spectral gap, \textit{i.e.}, the inverse Poincar\'e constant, of a convenient diffusion operator. However, two obstructions may occur within this framework: on the one hand the weight may not be explicitly computable and on the other hand enforcing the existence of those saturating functions reduces drastically the scope of admissible weights. For instance the latter approach  fails when considering linear functions associated to $\weightId$ as in Section \ref{sect:linear}, which are not always square integrable with respect to heavy-tailed distributions.  

In this part, we compute explicitly the Poincar\'e constant for more general weights. In particular we have in mind some examples for which the weighted Poincar\'e inequality does not admit saturating functions, so that Theorem \ref{teo: weight} cannot be applied. To address this problem, the approach we adopt is based on the intertwining technique emphasized in \cite{bj,bjm} in the one-dimensional case. Recall that a given probability measure $\mu \in \mathcal{P} (a,b)$ has density $\rho \in \funClass$ which takes positive values at the boundary and the notation $L_w$ stands for the (self-adjoint extension of the) operator defined on $\C_{N,w}^\infty(a,b)$ by 
$$
L_w f = w f'' + \left( w' + w (\log \rho )' \right) \, f',   
$$
where $\weight\in \funClassWeight$ is a weight function. 

Let us state first Theorem 4.2 in \cite{bj} but adapted to the present context. 
\begin{theo}
\label{theo:bj}
Assume that there exists some smooth function $h$ with non vanishing derivative on $(a,b)$ such that the function 
\begin{equation}
\label{eq: intertwining function chen}
    M_{w,h} := \frac{(-L_w h )' }{h'}, 
\end{equation}
is bounded from below on $(a,b)$ by some positive constant. Then we have the following weighted Poincar\'e inequality: for every centered function $f\in H^1(\mu,\weight)$, 
$$
\int_{a}^b f ^2 \, d \mu \leq \int_a^b \frac{\weight \, (f') ^2} {M_{w,h} } \, d\mu .
$$
In particular it holds  
\begin{equation}
\label{eq: intertwining upper bound}
C_P(\mu,\weight)
\leq \frac{1}{\underset{(a,b)}{\inf} \, M_{w,h}} .
\end{equation}
\end{theo}
The formula \eqref{eq: intertwining upper bound} can be seen as a generalization to the operator $L_w$ with diffusion constant $w$ of the famous Chen-Wang result \cite{Chen} on the spectral gap established by a coupling technique. As such, it has already been used in the Gaussian and generalized Cauchy cases to derive the exact expression of the Poincar\'e constants for some specific weights when there is no saturating function for the underlying weighted Poincar\'e inequalities, cf. \cite{bjm} (in the generalized Cauchy case, it corresponds to parameters $\beta \in (1/2,3/2]$, the range $\beta >3/2$ being recovered directly in Table \ref{table: weights}). In particular, an interesting choice of functions $h$ (or rather $h'$ since only $h'$ and its derivatives appears in the formula) parametrized by $\veps \in \R$ is $h_\veps' = \rho ^{-\veps} / \weight$ so that we have on $(a,b)$, 
\begin{equation}
\label{eq: intertwining function Joulin Bonnefont Ma}
M_{w,\veps} := M_{w, h_\veps} = (1-\veps) \, \weight \left(\veps\left((\log \, \rho )'\right)^2- (\log \, \rho)''\right).
\end{equation}
Then by \eqref{eq: intertwining upper bound} we obtain
\begin{equation}
\label{eq: eps intertwining upper bound}
C_P(\mu,\weight)
\leq \frac{1}{\underset{\veps\in \R}{\sup} \,  \underset{(a,b)}{\inf} \, M_{w, \veps}}.
\end{equation}

On the one hand when the weight is prescribed \textit{a priori}, the estimate \eqref{eq: eps intertwining upper bound} is convenient in many situations. On the other hand when we look also at some convenient weight $w$, we proceed as follows: compute first the right-hand-side of \eqref{eq: intertwining function Joulin Bonnefont Ma} (without $w$) which is expected to be non-negative for some relevant parameter $\veps$, and then choose the desired weight $w$ in such way that $M_{w,\veps}$ is bounded from below by some positive constant $c$. If for some reasons we suspect that $c = C_P(\mu,w)$, then we try to find some sequence of centered functions $(f_\eta)$ in $H^1 (\mu,w)$, indexed by some parameter $\eta$, such that the Rayleigh quotient $\int_{a}^b f_\eta ^2 \, d \mu / \int_{a}^b w \, (f_\eta') ^2 \, d \mu$ converges to $c$ as $\eta$ converges to some key value $\eta^\star$. In practice the sequence $(f_\eta)$ and the key value $\eta^\star$ are chosen such that the limiting function does not belong to $H^1 (\mu,w)$.

As announced, let us observe what happens in some classical and less classical examples, for which we are able to derive the exact value of the Poincar\'e constant although there is no function saturating the weighted Poincar\'e inequality.

\begin{Ex}
\rm{Considering the exponential distribution $\mu_\gamma$ with parameter $\gamma >0$, whose density is given on $\R^+$ by $\rho (x)  = \gamma e^{-\gamma x}$, we have seen in Section \ref{sect:non_vanishing} that the choice $w (x) = \weightId (x) = x/\gamma $ is related to linear eigenfunctions of the so-called Laguerre operator and yields $C_P (\mu_\gamma, \weightId ) = 1$, see for instance \cite{BGL}. In the unweighted case, there is no function saturating the Poincar\'e inequality, although we are able to identify the corresponding Poincar\'e constant: by \cite{bob_ledoux_expo}, it is well known that $C_P(\mu_\gamma , 1) = 4/\gamma^2$ (from a spectral point of view, it corresponds to the bottom of the essential spectrum of the devoted operator) and the classical approach to obtain the inequality $C_P(\mu_\gamma,1)\leq 4/\gamma^2$ uses a specific integration by part formula satisfied by $\mu_\gamma$. Here we are able to recover the same bound directly with \eqref{eq: eps intertwining upper bound}. Indeed, in this case the function $M_{1,\veps}$ is constant on $(0,\infty)$ and has the value $\gamma^2\veps(1-\veps)$. Optimizing with respect to $\veps$ yields 
\[
C_P(\mu_\gamma,1) \leq \frac{1}{\underset{\veps \in \R}{\sup} \, \gamma^2\veps(1-\veps)} = \frac{4}{\gamma^2}.
\]
The converse inequality is proved by considering a family of centered functions $f_\eta(x)=e^{\eta x} - \gamma / (\gamma-\eta)$, with $\eta<\gamma/2$ so they belong to $H^1(\mu_\gamma,1)$: we have after some computations,
\[
\frac{\int_0 ^{\infty} f_\eta ^2 \,  d\mu_\gamma}{\int_{0}^\infty (f_\eta ')^2 \, d \mu_\gamma} = \frac{(\eta-\gamma)^2-\gamma(\gamma - 2\eta)}{\eta ^2 (\eta-\gamma)^2},
\]
and taking above the limit $\eta\rightarrow \gamma/2$ entails by \eqref{eq: Courant} the inequality $C_P(\mu_\gamma,1) \geq 4/\gamma^2$.
}\end{Ex}

\begin{Ex}
\rm{In a somewhat similar unweighted context, we consider the generalized logistic (or skew-logistic) distribution $\mu_\alpha $ of parameter $\alpha >0$ on $\R$, namely with Lebesgue density
$$
\rho (x) = \frac{\alpha e^{-x}}{(1+e^{-x}) ^{\alpha+1}}, \quad x\in \R. 
$$
Similarly to the previous example, this probability measure is also log-concave on the real line, \textit{i.e.}, $- \log \rho $ is a convex function on $\R$. When $\alpha = 1$ we deal with the classical logistic distribution, which can be seen as a regularized version around the origin of the Laplace  distribution on the real line since $\rho (x) \sim e^{-\vert x\vert }$ as $\vert x\vert \to \infty$. In this case, the authors in \cite{barthe_bianchini_colesanti} proved that $C_P (\mu_1 ,1) = 4$ but there is no corresponding function saturating the Poincar\'e inequality. Let us recover this value from \eqref{eq: eps intertwining upper bound} and even prove that $C_P (\mu_\alpha ,1) = 4/\alpha^2$ for all $\alpha \in (0,1]$ (the case $\alpha >1$ could be addressed as well, but we are only able to find the bounds $4/\alpha^2 \leq C_P (\mu_\alpha,1)  \leq 4$, as suggested by the computations below). Given some $\varepsilon \in \R$, we have for all $x\in \R$,
$$
M_{1,\veps} (x) = \veps(1-\veps)+\frac{1 - \veps }{(1+e^x)^2} \, \left((\alpha+1)(1-2\veps) e^x +\veps  (\alpha^2-1)\right). 
$$
Choosing $0 \leq \veps \leq 1/2$ entails that $M_{1,\veps}$ starts being increasing, then reaches its maximum and becomes decreasing. In particular this implies that
\begin{equation*}
    \inf_{x\in \R}M_{1,\veps} (x) = \min\left \{\lim_{x\rightarrow -\infty}M_{1,\veps}(x),\lim_{x\rightarrow \infty}M_{1,\veps}(x)\right \}
    = \veps(1-\veps) \min \{ \alpha^2,1 \} = \veps(1-\veps) \alpha^2.
\end{equation*}
Optimizing then in $\veps$ leads by (\ref{eq: eps intertwining upper bound}) to the desired inequality $C_P (\mu_\alpha,1) \leq 4/\alpha^2$. 
To get the reverse inequality, we choose in \eqref{eq: Courant} the centered functions $f_\eta(x)=(1+e^{- x})^\eta- 1/(\alpha-\eta)$, with $\eta <\alpha/2$ to ensure that $f_\eta\in H^1(\mu_\alpha ,1)$. After computations, it yields
\begin{align*}
\frac{\int_\R f_\eta ^2 \, d\mu_\alpha}{\int_\R (f_\eta')^2 \, d\mu_\alpha }  = \frac{\frac{\alpha}{\alpha-2\eta} - \frac{\alpha^2}{(\alpha-\eta)^2}}{\frac{2 \alpha  \eta^2}{(\alpha-2\eta)(\alpha+1-2\eta)(\alpha+2-2\eta)}} 
 = \frac{(\alpha+1-2\eta)(\alpha+2-2\eta) }{2  (\alpha-\eta)^2},
\end{align*}
and taking the limit $\eta\rightarrow \alpha/2$, we thus obtain $C_P(\mu_\alpha ,1)\geq 4/\alpha^2$.} 
\end{Ex}

\begin{Ex} \rm{Let us concentrate now on the Pareto distribution, which is a (one-sided) heavy-tailed 
probability measure
somewhat similar to the generalized Cauchy distribution. Given two parameters $z>0$ and $\alpha >0$, denote $\mu_{z,\alpha}$ the probability measure with Lebesgue density on $[z,\infty)$ defined by $x\mapsto  \alpha z^\alpha / x^{\alpha+1}$. We show that $C_P(\mu_{z,\alpha},w) = 4 /\alpha^2$ when the weight is $w(x)=x^2$ (in particular it does not depend on the parameter $z$, as expected when using a trivial scaling argument). Hence we can assume in the sequel $z=1$. Note that it differs from the weight emphasized in Table \ref{table: weights}. By \eqref{eq: intertwining function Joulin Bonnefont Ma} we have for all $\veps \in \R$, \[ 
M_{w,\veps }(x) = (1-\veps) \, x^2 \, \left(  \frac{\veps(\alpha+1)^2-(\alpha+1)}{x^2} \right) , \quad x > 1. \] To ensure the positivity of $M_{w,\veps}$, we choose $\veps \in (1/(\alpha+1), 1)$. Optimizing then in $\veps$ and using \eqref{eq: eps intertwining upper bound} yields \[ C_P(\mu_{1,\alpha},\weight) \leq \frac{4}{\alpha^2}. \] To get the converse inequality, choose the centered functions $f_\eta(x) = x^\eta- \alpha /(\alpha - \eta)$, with $\eta<\alpha/2$ (so that they belong to $H_1(\mu_{1,\alpha},\weight)$). Indeed after some computations, we get 
$$ 
\frac{\int_{1}^\infty f_\eta ^2 \, d \mu_{1,\alpha} }{\int_{1}^\infty w \, (f_\eta') ^2 \, d \mu_{1,\alpha}}
= \frac{\frac{\alpha}{\alpha-2\eta}-\left(\frac{\alpha}{\alpha-\eta}\right)^2}{\frac{\alpha\, \eta^2}{\alpha-2\eta}} =\frac{1}{(\alpha-\eta)^2}, 
$$ and since the right-hand side converges to $4/\alpha^2$ when $\eta \rightarrow \alpha/2$, this proves the desired inequality $C_P(\mu_{1,\alpha},w) \geq 4/\alpha^2$. Again in this situation, the weighted Poincar\'e inequality  does admit any saturating function. Indeed, if it was the case, then there would be some smooth non-null centered function $f\in H^1 (\mu_{1,\alpha}, w)$ such that
\[\left\{\begin{array}{l}
     L_w f(x)=x^2f''(x)+(1-\alpha)x f'(x)=-\frac{\alpha^2}{4}f(x),\quad x>1, \\
     f'(1) = 0.
\end{array} \right.\]
Then one can prove that such a centered function is given by 
$$
f(x)= x^{\frac{\alpha}{2}} \, \left( 1 - \frac{\alpha}{2} \, \log(x) \right) , \quad x\geq 1 , 
$$
but it is not an eigenfunction since $f\notin L^2(\mu_{1,\alpha})$.
}
\end{Ex}

\begin{Ex}
\rm{We consider on $(0,1)$ the probability measure $\mu_\alpha$ with density $\rho(x)= \alpha x^{\alpha-1}$ ($\alpha >0$) and the asymmetric weight $\weight (x) =x^2 (1-x^{\alpha/2})$, $x\in [0,1]$. Although the density might vanish at the boundary (or even be not defined at 0), it causes no trouble for the forthcoming analysis. The same remark holds for the next example. Notice that when $\alpha=1$, $\mu_\alpha$ stands for the uniform distribution. This time the identity \eqref{eq: intertwining function Joulin Bonnefont Ma} is not relevant since $\inf M_{\weight,\veps}\leq 0$ for all $\veps \in \R$. However we are able to prove that $C_P (\mu_\alpha, w) = 4/\alpha^2$ by using \eqref{eq: intertwining function chen} with some convenient function $h'$. More precisely, since we have 
$$
L_w f(x)=x^2(1-x^{\alpha/2})f''(x)+x\left(\alpha+1-\left(\frac{3 \alpha}{2} +1\right)x^{\alpha/2}\right)f'(x), \quad x\in (0,1),
$$
then considering $h'(x) = x^{-\delta}$ with $\delta \in \R$ yields
$$
M_{w,h} (x) = -(\alpha-\delta+1)(1-\delta)+ \left( 1+\frac{\alpha}{2} - \delta \right)\left(\frac{3}{2}\alpha-\delta+1\right)x^{\alpha/2}, \quad x\in (0,1),
$$
which is non-negative as soon as $\delta \in [1, 1+\alpha/2]$. Choosing then $\delta= 1+ \alpha /2$ entails by \eqref{eq: intertwining upper bound} the bound $C_P(\mu_\alpha,\weight)\leq 4/\alpha^2$. \\ To obtain the lower bound we consider the function $f_\eta(x)=x^{-\eta}-\alpha/(\alpha-\eta)$, with $\eta<\alpha/2$ (so that $f_\eta\in H^1(\mu_\alpha,\weight)$) and compute:
\[\frac{\int_0^1 f_\eta ^2 \, d\mu_\alpha}{\int_0^1 \weight \, (f_\eta ')^2 \, d \mu_\alpha} =\frac{\frac{\alpha}{\alpha-2\eta}-\frac{\alpha^2}{(\alpha-\eta)^2}}{\frac{\alpha^2 \eta^2 }{2(\alpha-2\eta)\left(\frac{3}{2}\alpha-2\eta\right)} } = \frac{3 \alpha-4\eta}{\alpha(\alpha-\eta)^2}.\]
By taking the limit $\eta\rightarrow \alpha/2$ we establish that $C_P(\mu_\alpha,\weight)\geq 4/\alpha^2$. Note however that an associated eigenfunction does not exist since the non-null centered function $f(x)=x^{-\alpha/2} - 2$, which solves the equation 
$$
L_w f(x) = - \frac{\alpha^2}{4} \, f(x) , \quad x\in (0,1),
$$
does not belong to $L^2(\mu_\alpha)$.}
\end{Ex}

\begin{Ex}
\rm{
As a last example we consider the symmetric beta distribution on $(-1,1)$ whose density is $\rho (x) = (1-x^2)^{\beta-1} /Z_\beta$, where $\beta >0$ and $Z_\beta$ is the normalization constant $Z_\beta = \int_{-1}^1 (1-x^2)^{\beta-1} dx$. Using the spectral approach, it is known that $\mu_\beta$ satisfies a weighted Poincar\'e inequality with weight $w_0 (x) = 1-x^2$ and Poincar\'e constant $C_P (\mu_\beta , w_0) = 1/2\beta$. In particular linear functions (through Jacobi polynomials of degree 1), are the saturating functions, see \textit{e.g.} \cite{BGL}. However we wonder if such an inequality still hold with another weight for which the spectral analysis of the underlying operator does not give immediately the expression of the Poincar\'e constant. Actually, we are able to prove below that a weighted Poincar\'e inequality holds with weight $w (x) = (1-x^2)^2$ and corresponding Poincar\'e constant $C_P (\mu_\beta , w) = 1/\beta ^2$ for $\beta \in (0,1]$ (for $\beta >1$ we obtain directly $C_P (\mu_\beta , w) = 1/(2\beta -1)$ since the increasing centered function $x\mapsto x/\sqrt{1-x^2} \in L^2 (\mu_\beta)$ is an eigenfunction associated to the eigenvalue $2\beta-1$ of the operator $-L_w$ below).
First, as in the previous example, the identity \eqref{eq: intertwining function Joulin Bonnefont Ma} does not allow us to get the desired result so that let us rather use the equation \eqref{eq: intertwining function chen} with some other convenient test function $h$. Since the associated operator is given for all $x\in (-1,1)$ by  
\[
L_w f(x)=(1-x^2)^2 f''(x)-2(\beta+1)x(1-x^2)f'(x),
\]
we obtain by plugging in \eqref{eq: intertwining function chen} the function $h'(x)=(1-x^2)^{\delta}$, 
\[
M_{\weight,h} (x) = 2(\delta+\beta+1) (1 - (3+2\delta )x^2 ),
\]
whose minimum on $[-1,1]$ is reached at $x = \pm 1$ and positive as soon as $\delta \in (-1-\beta, -1)$. Choosing finally $\delta = -1-\beta/2$ implies by \eqref{eq: intertwining upper bound} the upper bound $C_P(\mu_\beta,\weight) \leq 1/\beta^2$. } \\
On the other hand, we consider the function $f_\eta(x)=(1-x^2)^\eta- Z_{\eta+\beta} / Z_\beta$ (with $\eta > -\beta/2$ to ensure $f_\eta\in H^1(\mu_\beta,\weight)$) and compute:
\begin{align*}
\frac{\int_{-1}^1 f_\eta ^2 \, d\mu_\beta}{\int_{-1}^1 \weight \, (f_\eta ')^2 \, d \mu_\beta} & = \frac{\frac{Z_{2\eta+\beta}}{Z_\beta} - \left( \frac{Z_{\eta+\beta}}{Z_\beta} \right) ^2}{\frac{2\eta^2}{2\eta + \beta} \times \frac{Z_{2\eta + \beta +1}}{Z_\beta}}.
\end{align*}
Then the continuity of $a \mapsto Z_a$ on $(0,\infty)$ together with an integration by parts leading to the following identity: for all $a>0$,
$$
Z_{a+1} = \frac{2a}{1+2a} \, Z_a,
$$
entails that we finally have 
$$
\frac{\int_{-1}^1 f_\eta ^2 \, d\mu_\beta}{\int_{-1}^1 \weight \, 
(f_\eta ')^2 \, d \mu_\beta} \underset{\eta \to - \beta /2}{\sim} \frac{1}{\beta^2}.
$$
Hence we get the reverse inequality $C_P(\mu_\beta,\weight)\geq 1/\beta^2$. However for every $\beta \in (0,1]$ there is no associated eigenfunction. Indeed, the related problem is to find some smooth non-null centered function $f\in H^1(\mu_\beta,\weight)$ such that
$$
L_w f (x) = - \beta^2 f(x) , \quad x\in (-1,1).
$$
Rewriting $f(x)=(1-x^2)^{-\beta/2}g(x)$, this is equivalent to find some solution $g$ such that the following Legendre equation equation holds: for all $x\in (-1,1)$, 
\[
(1-x^2)g''(x)-2xg'(x)=-\beta(\beta+1) g(x).
\]
For $\beta \in (0,1)$, solutions are linear combinations of Legendre functions of the first and second kind, the first one admitting finite limits at $\pm 1$ whereas the second one, denoted $Q_\beta$, is singular at the endpoints $\pm 1$, see \cite{Handbook}. In particular we have the asymptotics $Q_\beta (x) \underset{x\to 1}{\sim} - \log (1-x) /2 $ so that the resulting function  $f(x)=(1-x^2)^{-\beta/2}g(x)$ does not belong to $L^2(\mu_\beta)$. Finally for $\beta = 1$, solutions to the Legendre equation are linear (\textit{i.e.}, degree 1 Legendre polynomials) but also in this case we have $f \notin L^2(\mu_\beta)$.
\\ 
To conclude, let us prove that for every $\beta>0$, the exponent $2$ in $\weight(x)=(1-x^2)^2$ is the largest one leading to a weighted Poincar\'e inequality. Suppose by contradiction that there exists some $\veps>0$ such that $\weight_\veps(x)=(1-x^2)^{2+\veps}$ is an admissible weight. By \eqref{eq: Courant} applied to the function $f_\eta(x)=(1-x^2)^\eta- Z_{\eta+\beta} /Z_{\beta}$, $\eta>-\beta/2$, it follows that
\[C_P(\mu,\weight_\veps)\geq  \frac{Z_{2\eta+\beta}-\frac{Z_{\eta+\beta}^2}{Z_\beta}}{4\eta^2\left(Z_{2\eta+\beta+\veps}-Z_{2\eta+\beta+\veps+1}\right)},\]
which tends to infinity for each fixed $\veps >0$ as $\eta\rightarrow -\beta/2$, leading thus to a contradiction.
\end{Ex}

\section{Link with global sensitivity analysis}
\label{sect:GSA_link}
\label{Link with Global Sensitivity Analysis}
Let us turn our attention to the consequences of weighted Poincar\'e inequalities for sensitivity analysis. First we provide the proof of the inequality involving total Sobol and weighted DGSM indices defined earlier in the introduction. Subsequently we deal with the equality case and emphasize a stability condition ensuring the sharpness of the upper bounds and introduce data-driven weights, together with a uniform consistency result. At the end, we address the Poincar\'e chaos approach to produce lower bounds for total Sobol indices.

In the whole section, we consider a random vector $X=(X_1,\dots,X_d) \in \R^d$ of independent input variables and the output $f(X) \in L^2$, where $f:\R^d \to \R$ is some function referring to the model. Recall that for a set $I\subset \{1,\dots,d\}$ the notation $X_{I}$ stands for the random vector defined by the variables $X_{j}$, $j\in I$. As defined in the introduction, the total Sobol indices are given by
\begin{equation*}
S_i^\tot (f(X)) = \frac{\Var(f_i^\tot(X))}{\Var(f(X))} ,
\end{equation*}
where 
\begin{equation} 
\label{eq:fiTotFormula}
f_i^\tot(X) = \sum_{I\ni i}f_I(X_I) = f(X) - \Esp \left[ f(X) \br X_{-i}\right],
\end{equation}
the random vector $X_{-i}$ of dimension $d-1$ being defined by $X_{-i}:=X_{\{1,\dots,d\}\setminus\{i\}}$. Note that the second equality in \eqref{eq:fiTotFormula}, which is important in the forthcoming analysis, comes from the fact that if $I$ is a superset of $\{i\}$, then by the assumptions of the Sobol-Hoeffding decomposition,
$$ 
\Esp[f_I(X_I) \vert X_{-i}] = 
\Esp[f_I(X_I) \vert X_{I \setminus \{i \}}] = 0.
$$

Finally, given $i\in \{ 1, \ldots, d \}$ and some convenient non-negative weight  $\weight_i$, we further assume that $\left( \weight_i(X_i) \right)^{1/2} \frac{\partial f}{\partial x_i}(X) \in L^2$. Then, the weighted DGSM index related to the distribution of the input $X_i$ is defined as 
$$
\DGSMi (f(X)) = \Esp\left[\weight_i(X_i)\left(\frac{\partial f}{\partial x_i}(X) \right)^2\right] .
$$
To simplify the notation, we write $S_i^\tot$ for $S_i^\tot (f(X))$ and $\nu_{i,\weight_i}$ for $\nu_{i,\weight_i} (f(X))$.

\subsection{Upper bounds for total Sobol indices}
\label{sect: upper bound}
The following proposition deals with the inequality announced in the introduction, linking total Sobol and weighted DGSM indices. It generalizes the result presented in \cite{lamboni} for classical Poincar\'e inequalities.
\begin{prop}
\label{prop: upper bound}
Given $i \in \{ 1,\ldots, d \}$, assume that the distribution $\mu_i$ of the input variable $X_i$ belongs to $\probClassi$. Let $\weight_i \in \funClassWeighti$ be a weight function and suppose that $\mu_i$ satisfies the following
weighted Poincar\'e inequality:  for all centered function $g\in H^1(\mu_i,\weight_i)$,
\begin{equation}
\label{eq: poincar\'eGSA}
\Esp\left[g(X_i)^2\right] \leq C_P(\mu_i,\weight_i) \, \Esp\left[\weight_i(X_i)g'(X_i)^2\right] .
\end{equation}
We also assume that $\left( \weight_i(X_i) \right)^{1/2} \frac{\partial f}{\partial x_i}(X) \in L^2$ where $f(X) \in L^2 $. Then, 
\begin{equation}
  \label{eq: prop upper bound}
      S_i^\tot\leq C_P(\mu_i,\weight_i) \, \frac{\DGSMi}{\Var(f(X))}.
  \end{equation}
\end{prop}
\begin{proof}
For all $x_{-i} \in \prod_{j\neq i} [a_j, b_j]$, let $g _{x_{-i}}$ be the one-dimensional function 
$$g _{x_{-i}} : x_i \in [a_i, b_i] \mapsto f_i^\tot(x).$$
Notice that, from \eqref{eq:fiTotFormula}, we can write $g _{x_{-i}}(x_i) = f(x) - \Esp[f(X)\br X_{-i}=x_{-i}]$. Under this form, it is clear that $g _{x_{-i}}$ is a centered function of $L^2(\mu_i)$ and verifies
$$
g_{x_{-i}}'(x_i) = \frac{\partial f}{\partial x_i}(x).$$
Under the assumptions on $f$, this latter equality shows that $g _{x_{-i}}$ belongs to $H^1(\mu_i, \weight_i)$. Then the weighted Poincar\'e inequality \eqref{eq: poincar\'eGSA} applied to $g_{x_{-i}}$ gives
$$ \Esp \left(f_i^\tot(X)^2 \vert X_{-i} = x_{-i} \right)
\leq C_P(\mu_i,\weight_i) \, \Esp  \left[ \left. \weight_i(X_i) \left( \frac{\partial f}{\partial x_i}(X) \right)^2 \right\vert X_{-i} = x_{-i} \right].$$
Integrating over $x_{-i}$ with respect to the product measure $\otimes_{j \neq i} \mu_j$ and using the fact that $f_i^\tot$ is centered, we get
$$ \Var \left(f_i^\tot(X) \right)
\leq C_P(\mu_i,\weight_i) \, \Esp  \left[ \weight_i(X_i) \left( \frac{\partial f}{\partial x_i}(X) \right)^2  \right].
$$
Dividing by $\Var(f(X))$ concludes the proof.
\end{proof}

\subsection{Case of equality in the upper bound and stability}
\label{Case of equality in the upper bound and stability}
Looking carefully at the proof of Proposition \ref{prop: upper bound} above, we observe that a sufficient condition ensuring the equality in \eqref{eq: prop upper bound} 
is that the one-dimensional function $x_i \mapsto f_i^\tot(x)$ saturates the weighted Poincar\'e inequality \eqref{eq: poincar\'eGSA} for each fixed $x_{-i} \in \R^{d-1}$. Since the space of saturating functions is one-dimensional (it corresponds to the eigenspace related to the spectral gap of the underlying diffusion operator), the function $f_i^\tot$ is then required to be of the following form: 
$$
f_i^\tot (x) = g_i (x_i) v (x_{-i}), \quad x\in \R^d,
$$
where $g_i \in H^1(\mu_i,\weight_i)$ is some one-dimensional function saturating the weighted Poincar\'e inequality \eqref{eq: poincar\'eGSA} (in particular it is centered with respect to $\mu_i$) and $v_i$ is a function that only depends on $x_{-i}$. In this case the function $f$ rewrites as 
$$
f(x) = u_i (x_{-i}) + g_i(x_i) v_i(x_{-i}), \quad x\in \R^d,
$$
with $u_i (x_{-i}) = \Esp [f(X) \br X_{-i} = x_{-i}]$. The next proposition further provides a stronger result in terms of stability of the inequality \eqref{eq: prop upper bound}, at least when the one-dimensional function considered is close to the function saturating the Poincar\'e inequality \eqref{eq: poincar\'eGSA} in the space $H^1 (\mu_i,w_i)$ (it corresponds below to the case $\veps > 0$).

\begin{prop}
\label{prop: weightstability}
Under the notation and assumptions of Proposition \ref{prop: upper bound}, let fix $i\in \{ 1,\ldots, d \}$ and let $f$ be of the form $f(x) = u_i(x_{-i}) + h_i (x_i) v_i(x_{-i})$, where $u_i, v_i$ are functions depending only on $x_{-i}$, such that $u_i (X_{-i}) , v_i (X_{-i}) \in L^2$, and $h_i \in H^1(\mu_i,\weight_i)$ is centered with respect to $\mu_i$. Let further assume that there exists $\veps \geq 0$ such that 
\[
\Esp\left[\weight_i(X_i) \left( h_i'(X_i)-g_i'(X_i) \right) ^2\right]\leq \veps,
\]
where $g_i$ is a function saturating the weighted Poincar\'e inequality \eqref{eq: poincar\'eGSA}. Then we have the following stability result:
\[
0 \leq   C_P(\mu_i,\weight_i)\frac{\DGSMi}{\Var(f(X))}-S_i^\tot\leq C_P(\mu_i,\weight_i) \, \frac{\Esp\left[v_i^2(X_{-i})\right]}{\Var(f(X))} \, \veps.\]
In particular, if $f$ has the form $f(x)= u_i(x_{-i}) + g_i(x_i) v_i(x_{-i})$, then \eqref{eq: prop upper bound} is an equality. 
\end{prop}
\begin{proof}
Denote $V$ the symmetric bilinear form acting on the space of centered functions in $H^1 (\mu_i, w_i)$ as follows: 
$$
V(\varphi,\psi ) = C_P(\mu_i,\weight_i)\,\Esp\left[ w_i (X_i) \, \varphi'(X_i) \, \psi '(X_i) \right] - \Esp\left[ \varphi(X_i) \, \psi (X_i) \right] . 
$$
By the weighted Poincar\'e inequality \eqref{eq: poincar\'eGSA}, the induced quadratic form is non-negative. Moreover we have 
$$
V(h_i-g_i , h_i - g_i ) = V(h_i, h_i) + V(g_i , g_i ) - 2\, V(h_i , g_i).
$$
Since the function $g_i$ saturates \eqref{eq: poincar\'eGSA}, we have $V(g_i , g_i ) = 0$ and also  $V(h_i , g_i) = 0$, using the variational identity \eqref{eq: variational formulation}. Hence we get 
\begin{equation} \label{eq: Vh_maj}
0 \leq \,  V(h_i , h_i) = V(h_i-g_i , h_i - g_i ) 
\leq C_P(\mu_i,\weight_i)\,\veps.
\end{equation}
Notice that $\frac{\partial f}{\partial x_i}(x) = h_i'(x_i) \, v_i(x_{-i})$ and $f_i^\tot(x) = h_i(x_i)v_i(x_{-i})$.
Then, by independence of $X_i$ and $X_{-i}$, we obtain:
$$
C_P(\mu_i,\weight_i) \, \frac{\DGSMi}{\Var(f(X))} - S_i^\tot = \frac{\Esp\left[v_i^2(X_{-i})\right]}{\Var(f(X))} \, V(h_i,h_i).
$$
Combining with \eqref{eq: Vh_maj} concludes the proof.
\end{proof}

Although $w_i$ can be any general weight in $\funClassWeighti$ \textit{a priori}, the requirement that the function $g_i$ saturates the weighted Poincar\'e inequality \eqref{eq: poincar\'eGSA} enforces $w_i$ to be proportional to the weight $w_{g_i}$ constructed according to Theorem \ref{teo: weight}. Indeed, it is not difficult to prove that necessarily $w_i = w_{g_i}/ C_P (\mu_i,w_i)$.

\subsection{Data-driven weights for monotonic main effects}
\label{Data-driven weights}

Important quantities arising in GSA are the first-order indices, or main effects, which allow to understand the influence of each variable individually over the model. For each $i\in \{ 1, \ldots, d \}$, the main effect $f_i (X_i)$ (or simply $f_i$, using a slight abuse of language) of the input variable $X_i$ is defined as 
$$
f_i (X_i) = \Esp \left[ f(X) \br X_i \right] - \Esp \left[ f(X)\right].
$$
By definition of the conditional expectation, $f_i(X_i)$ is the best centered $L^2$ approximation of $f(X)$ by a one-dimensional function depending on $X_i$. It is thus reasonable to use $f_i$ for building $w_i$. If $f_i$ belongs to $\funClassTwoi$, implying that $f_i$ is strictly monotonic, a natural candidate for $w_i$ is $w_{f_i}$ since by Theorem~\ref{teo: weight}, $f_i$ saturates the corresponding weighted Poincar\'e inequality. By Proposition~\ref{prop: weightstability}, this further implies that the total Sobol indices will be perfectly approximated by weighted DGSM ones involving the weight $w_{f_i}$ for functions of the form $f(x) = u_i(x_{-i}) + h_i (x_i) v_i(x_{-i})$. Indeed, then we necessarily have $f_i (x_i) = h_i (x_i) \, \mathbb{E} [v_i (X_{-i})]$, and (provided that $\mathbb{E} [v_i (X_{-i})] \neq 0$) $h_i$ is proportional to $f_i$, corresponding to an equality case in Proposition \ref{prop: upper bound}. \\
In practice, however, $f_i$ is unknown and must be estimated. The idea is to use some convenient pointwise estimator $\weightSTIMi$ of $\weight_{f_i}$, \textit{i.e.}, $\weightSTIMi(x_i)$ is an estimator of $\weight_{f_i}(x_i)$ for all $x_i \in [a_i,b_i]$. A natural candidate for this data-driven weight $\weightSTIMi$ is obtained by first approximating $f_i$ with a pointwise strictly monotonic estimator $\fSTIMi$, almost surely centered with respect to $\mu_i$, and then considering $\weightSTIMi=\weight_{\fSTIMi}$. The consistency of this statistical procedure is proved in the following proposition.
\begin{prop}
\label{prop:consistency}
Given $i \in \{ 1,\ldots, d \}$ and a finite interval $[a_i,b_i]$, we assume that $\mu_i \in \probClassi$ and that the main effect $f_i \in \funClassTwoi$ satisfies $f'_i(a)\neq 0$ and  $f'_i(b)\neq 0$. Moreover we assume that $\frac{\partial f }{\partial x_i}(X)\in L^2 $. \\
Let $(X_1,\ldots, X^\nn)$ be an i.i.d. sample generated from $X$ and consider an estimator $\fSTIMi$ constructed with respect to this sample. We assume the following (almost sure) hypothesis:
\begin{itemize}
\item[(i)] $\fSTIMi$ is centered with respect to $\mu_i$. 
\item[(ii)] $\fSTIMi \in \funClassTwoi$ satisfies $\fSTIMi '(a)\neq 0$ and  $\fSTIMi '(b)\neq 0$. 
\item[(iii)] $\fSTIMi$ converges uniformly to $f_i $ on $[a_i,b_i]$ as $n\rightarrow \infty$.
\item[(iv)] $\fSTIMi'$ converges uniformly to $f_i '$ on $[a_i,b_i]$ as $n\rightarrow \infty$.
\end{itemize}
Then almost surely, the data-driven weight $\weightSTIMi=\weight_{\fSTIMi}$ converges uniformly to $\weight_{f_i}$ on $[a_i,b_i]$ as $n\rightarrow \infty$. Moreover, the Monte-Carlo estimator 
\[
\widehat{\nu_{i}}=\frac{1}{n}\sum_{k=1}^n \widehat{\weight_{i}}(X_i^\kk)\abs{\frac{\partial f}{\partial x_i}(X^\kk)}^2\]
converges almost surely to $\nu_{i,\weight_{f_i}}$ as $n\rightarrow \infty$.
\end{prop}

\begin{proof}
Using (i) and (ii), the weight $\weightSTIMi = \weight_{\fSTIMi}$ is defined in Theorem \ref{teo: weight} for all $x_i\in [a_i, b_i]$ by
\begin{align}
\weightSTIMi(x_i)
=-\frac{1}{\fSTIMi'(x_i)\rho_i(x_i)}\int_{a_i}^{x_i} \fSTIMi(y)\rho_i(y)\,dy. \label{eq: weight gn}
\end{align}
Since a.s. $\fSTIMi$ converges uniformly to $f_i$ on $[a_i,b_i]$ as $n\to \infty$, it follows that a.s. the sequence $x_i \mapsto \int_{a_i} ^{x_i} \fSTIMi(y)\rho_i(y)\,dy$ also converges uniformly to $x_i\mapsto 
\int_{a_i} ^{x_i} f_i(y)\,\rho_i(y)\,dy$ on $[a_i,b_i]$. By (ii), $\fSTIMi$ and $f_i$ are $\C^1$ on the finite interval $[a_i, b_i]$. Since $\fSTIMi'$ tends uniformly to $f_i'$ and $\fSTIMi'$ is of constant sign, the function $\vert \fSTIMi' \vert $ is bounded from below by a positive constant that does not depend on $n$. The same argument holds for $\vert f_i '\vert $ as well. Using that $\rho_i$ is also bounded from below by some positive constant, this further implies that  $1/(\fSTIMi' \rho_i) $ converges uniformly to $1/(f_i' \rho_i) $ on $[a_i,b_i]$. This concludes the proof of the desired a.s. uniform convergence of $\weightSTIMi$. \\
Now the proof of the convergence of the Monte-Carlo estimator $\widehat{\nu_{i}}$ to the weighted DGSM index $\nu_{i,w_{f_i}}$ is straightforward: decomposing $\widehat{\nu_{i}}$ as
\begin{align*}
\scaleobj{.96}{
\displaystyle
    \widehat{\nu_{i}}=\frac{1}{n}\sum_{k=1}^n \left(\widehat{\weight_{i}}(X_i^\kk)-\weight_{f_i}(X_i^\kk)\right)\abs{\frac{\partial f}{\partial x_i}(X^\kk)}^2+\frac{1}{n}\sum_{k=1}^n \weight_{f_i}(X_i^\kk)\abs{\frac{\partial f}{\partial x_i}(X^\kk)}^2,
    }
\end{align*}
we observe on the first hand that the first term converge a.s. to 0 (since a.s. $\weightSTIMi$ converges uniformly to $\weight_{f_i}$ on $[a_i,b_i]$ as $n\rightarrow \infty$ and $\frac{\partial f}{\partial x_i}(X)\in L^2$) whereas on the other hand the second term converges a.s. to $\nu_{i,w_{f_i}}$ by the strong law of large numbers. The proof is now complete.
\end{proof}

Note that the assumptions of Proposition \ref{prop:consistency} are rather strong. A difficulty is to ensure  
that $\fSTIMi'$ can be bounded from below by a positive constant that does not depend on the sample size $n$. This is guaranteed by the uniform convergence of the derivatives $\fSTIMi'$. Notice also that Assumption (iii) may be replaced by
\begin{itemize}
    \item[(iii')] there exists some $c_i \in [a_i,b_i]$ such that $\fSTIMi (c_i)$ converges to $f_i (c_i)$ as $n\rightarrow \infty$
\end{itemize}
since (iii') and (iv) imply (iii). This is however a particular case, and we have chosen to keep here the most general assumption.

Note also that the relevance of the data-driven weight methodology depends crucially on the construction of the strictly monotonic $\fSTIMi$ which estimates the main effect $f_i$. In practice it is difficult to provide such a construction ensuring the monotonicity and the uniform consistency properties at the same time. See however the paper \cite{Neumeyer2007} which explains how to modify a given uniformly consistent estimator to obtain a monotonic version of it. In particular the uniform consistency property is preserved with the same rate of convergence.

To achieve the discussion about data-driven weights, we point out that the previous approach can not be applied in theory when the main effect $f_i$ is not monotonic. Indeed, the weight $\weight_{f_i}$ is not well defined and the required convergence properties of the estimator $\fSTIMi$ are not guaranteed. However in practice such a procedure can be adapted when the strictly monotonic estimator $\fSTIMi$ remains close to $f_i$ (in the weighted Sobolev space related to $\weightSTIMi$), so that the numerical results obtained with the data-driven weight $\weightSTIMi$ may be somewhat relevant in this context. 

\subsection{Approximations of total Sobol indices with Poincar\'e chaos expansions}
\label{chaos}

In the previous sections, our main objective was to bound from above the total Sobol indices with the weighted DGSM ones. However when we consider a
weight $w_i$ that does not vanish on the finite interval $[a_i,b_i]$, the associated Sturm-Liouville problem is regular and the underlying diffusion operator $L_{w_i}g = (\weight_i g' \rho_i)'/{\rho_i}$ admits a discrete spectral decomposition. Collecting all these decompositions for all $i \in \{ 1, \ldots, d \}$, such a decomposition considered on the product space leads to the so-called Poincar\'e chaos expansion and allows us to expand $\Var (f_i ^{\tot} (X))$. As such, truncating the expansion yields to lower bounds on the total Sobol index $S_i ^{\tot}$. Below we provide further details on this approach involving weights, which 
is a simple adaptation of the one presented in \cite{Chaos2,PoincareChaos} and related to classical Poincar\'e inequalities.

Before stating our desired expansion, we need to fix some elements. For each $i \in \{ 1,\ldots, d\}$, recall that each $\mu_i\in \probClassi$ stands for the distribution of the input variable $X_i$. Let $w_i\in \funClassWeighti$ which does not vanish at the boundaries. Denote $(\lambda_{i,n})_{n\in \N}$ the eigenvalues listed in increasing order of the diffusion operator $-L_{w_i}$. We have $\lambda_{i,0} = 0$ and $\lambda_{i,1} = 1/C_P (\mu_i,w_i)$ is the spectral gap. Let $(e_{i,n})_{n\in \N}$ be the associated orthonormal basis of eigenfunctions (recall that $e_{i,0}\equiv 1$), the normalization being understood in $L^2(\mu_i)$.

In higher dimension we consider the product measure $\mu = \mu_1 \otimes \dots  \otimes \mu_d$ of the input vector $X$. Letting $L^2(\mu)$ be the space of square-integrable functions on the cartesian product $\prod_{i=1}^d[a_i, b_i]$ with respect to $\mu$, we write the scalar product of two given functions $\varphi, \psi\in L^2(\mu)$ as $\langle \varphi,\psi\rangle = \int \varphi\,\psi\,d\mu$. Then the diffusion operator related to $\mu$ is the sum of the one-dimensional operators $-L_{w_i}$, each acting only on the $i$-th coordinate of a given multi-dimensional function, the other coordinates being fixed. Then the collection of all tensor-product functions $e_{\alpha}: x\mapsto \prod_{i=1}^d e_{i, \alpha_i}(x_i)$, with
$\alpha=(\alpha_1,\dots,\alpha_d)\in \N^d$ a multi-index, form an orthonormal basis of $L^2(\mu)$ and correspond to the eigenfunctions of this multi-dimensional diffusion operator, each eigenfunction $e_\alpha$ being related to the eigenvalue $\sum_{i=1}^d \lambda_{i,\alpha_i}$. Hence every function in $L^2 (\mu)$ admits a decomposition in terms of the functions $(e_{\alpha})_{\alpha \in \N^d}$, known as the Poincar\'e chaos expansion expansion (PoinCE). Applied to the function $f$, we have
$$
f = \sum_{\alpha \in \N^d} \langle f,e_{\alpha} \rangle \, e_{\alpha}.
$$
Since we have for all $x\in \prod_{i=1}^d[a_i, b_i]$,
$$
f_i^\tot (x) = f(x) - \int_{a_i} ^{b_i} f(x_1,\ldots,  x_i,  \ldots, x_d) \, \mu_i(dx_i) , 
$$
its PoinCE has the particular form  
$$
f_i ^{\tot} = \sum_{\substack{\alpha \in \N^d \\ \alpha_i\geq 1}} \langle f,e_{\alpha} \rangle \, e_{\alpha}.
$$
Hence the variance of the random variable $f_i ^{\tot} (X)$ can be expanded using Parseval's identity: 
\begin{align}
        \Var(f_i^\tot (X)) =   \sum_{\substack{\alpha \in \N^d \\ \alpha_i\geq 1}}\langle f, e_{\alpha} \rangle^2 .
        \label{eq: chaos approx}
    \end{align}
The advantage of using such a basis $(e_{\alpha})_{\alpha \in \N^d}$ rather than any other one resides in the fact that it provides an alternative expansion in terms of the derivatives of $f$, in the spirit of the one-dimensional identity \eqref{eq: variational formulation}: for all $i \in \{1,\ldots, d \}$ and all $n\in \N$, 
\[
\langle f,e_{i,n}\rangle=\frac{1}{\lambda_{i,n}} \, \langle \weight_i \, \frac{\partial f}{\partial x_i},e'_{i,n} \rangle.
\]
The next proposition summarizes these facts. The proof is somewhat similar to the one established in \cite{PoincareChaos}.
\begin{prop}
\label{prop: Chaos}
Assume that the weighted DGSM index $\DGSMi$ is finite. Then the Parseval's identity \eqref{eq: chaos approx} can be rewritten using derivatives as:
\begin{equation}
    \Var(f_i^\tot (X))=\sum_{\substack{\alpha \in  \N^d \\ \alpha_i\geq 1}}\frac{1}{\lambda^2_{i,\alpha_i}} \, \langle \weight_i \frac{\partial f}{\partial x_i},e'_{i,\alpha_i} \prod_{j\neq i}^d e_{j,\alpha_j} \rangle^2.
    \label{eq: chaos approx der}
\end{equation}
\end{prop}
In practice, the decompositions \eqref{eq: chaos approx} and \eqref{eq: chaos approx der} have to be appropriately truncated, leading to approximations of the Sobol index $S_i^\tot$. Thus, the key point is how to choose the truncation and then to estimate the various scalar products. Certainly, keeping only few terms in the Poincar\'e chaos expansion might provide relevant numerical results, as we will observe in Section \ref{sect:Appli} when dealing with our toy models and the real flood application. However it is not an easy task in full generality as soon as the dimension is large since the computational cost is high. A first step in this direction has been proposed in \cite{Chaos2} in the context of classical Poincar\'e inequalities using sparse regression. In particular this approach allows to compute simultaneously the scalar products involved in the decompositions \eqref{eq: chaos approx} and \eqref{eq: chaos approx der}. Dealing with the weighted case, we intend to apply the same methodology in a future work.

As a last remark, we mention that the data-driven approach emphasized in Section \ref{Data-driven weights} is hardly available at this stage when using the Poincar\'e chaos expansion, since in the latter case the data-driven weight $\weightSTIMi$ is required to be positive on $[a_i,b_i]$. In other words, it means that the derivatives of the estimator $\fSTIMi$ should vanish at the boundary, a property which is hard to guarantee in practice.  

\subsection{Summary and guidelines for choosing weights to GSA}
Before turning to Section \ref{sect:Appli} and the applications to various models arising in GSA, let us for completeness summarize our approach based on weighted Poincar\'e inequalities. Recall that for a given $i \in \{ 1,\ldots,d\}$ the total Sobol indices $S_i^\tot$ can be bounded from above by an expression involving the weighted DGSM index $\nu_{i,w_i}$ with an appropriate weight $w_i$ (see Section \ref{sect: upper bound}). Lower bounds on $S_i^\tot$ are also available when using the Poincar\'e chaos expansion, as emphasized in the last section. To do so, we study the spectral properties of the diffusion operator $L_{\weight_i}g = (\weight_i g' \rho_i)'/\rho_i$ depending on the weight $w_i$ which has to be constructed conveniently.

When no particular knowledge about the model $f$ is provided, the most information we can expect may be given by the one-dimensional $L^2$ centered approximation of $f(X)$ corresponding to the main effect $f_i (X_i)$ in the Sobol Hoeffding decomposition. Then we can tune the (one-dimensional) weight $\weight_i$ such that the first non-null eigenfunction $e_i$ of the operator $-L_{\weight_i}$ (constructed with respect to $w_i$) is close to $f_i$. Notice however that this approximation is constrained by the structural properties of $e_{i}$, in particular its monotonicity.

In many engineering problems, $f_i$ may be approximated by a linear function. This justifies the idea of choosing $\weightId$ 
(see Section \ref{sect:linear}) by default or its Gaussian approximation $\weightG$ (see Section \ref{sect:non_vanishing}). Nevertheless, as a first step of GSA, it is often possible to visualize the main effects and to estimate them from a sample of small size, see for instance \cite{GSA_book}. Thus it is natural to consider the data-driven weight $\weightSTIMi$ built from a strictly monotonic estimator $\fSTIMi$ of $f_i$ (see Section \ref{Data-driven weights}). As mentioned earlier, we illustrate this data-driven approach only for bounding from above the total Sobol indices, as their estimation from Poincar\'e chaos expansions requires that the derivatives of the estimator $\fSTIMi$ vanishes at the boundary, a property which is difficult to ensure in practice. A first idea in this direction could be to use a Gaussian process approach for constraints on boundedness, monotonicity and convexity, cf. e.g.  \cite{Bachoc}.

\section{Applications}
\label{sect:Appli}
\label{Applications}

\subsection{Numerical settings}
In this final part, we present numerical applications
of our results on two toy models and a more realistic one given by a flood case. To emphasize the role of the input variables $X_i$, $i\in \{1,\dots,d\}$, we rewrite the weights presented in Section \ref{Result for optimal weights} as $\weightIdi$, $\weightGi$ and $\weightUi$. In addition to $\weightIdi$ and $\weightSTIMi$, we consider also the non-vanishing weights $\weightUi$ and $\weightGi$ (see Section \ref{sect:non_vanishing}) for comparison purposes when bounding from above total Sobol indices. For the lower bounds using Poincar\'e chaos expansion (understood in this part as approximations), we employ only the weights $\weightGi$ and $\weightUi$, which ensure the existence of an orthonormal basis of eigenfunctions. 
We refer to truncations of the infinite sum in \eqref{eq: chaos approx} as derivative-free approximations (abbreviated Der-free) and to truncations of the one in \eqref{eq: chaos approx der} as derivative-based approximations (abbreviated Der-based). All our numerical results
have been done with the R software \cite{R}, and are compared with the ones obtained in the classical unweighted case. We now give more specific details. \smallskip 

\noindent \textit{Eigenvalue and eigenfunction estimation, and \PCE series truncation}
\hfill\\
The eigenvalues and eigenfunctions are computed using a weighted adaptation of the function
\Rnot{PoincareOptimal} from the package \Rnot{sensitivity} \cite{R_sensitivity}, based on a finite element discretization. 
Once this is done, the \PCE approximations can be computed. We approximate each $S_i^\tot$ with truncations of both sums \eqref{eq: chaos approx} and \eqref{eq: chaos approx der} by considering only few terms, according to the following subset of multi-indices:
\[\mathcal{A}_i=\biggr\{\alpha\in \N^d \brr \alpha_i\in\{1,2\},\,\sum_{j\neq i} \alpha_j\leq 1 \biggr\}.\]
In other words, a given $\alpha \in \mathcal{A}_i$ if and only if $\alpha_i\in\{1,2\}$ and there is at most one index $j\neq i$ such that $\alpha_j \leq 1$, all the other ones being null. As we will observe below, this selection is sufficient to obtain relevant results on the three models. \medskip 

\noindent \textit{Weight computation}
\hfill\\
Recall that in our approach, given $i\in \{1,\dots,d\}$, a weight $\weight_i$ is constructed from a strictly monotonic function $g$ which saturates the associated weighted Poincar\'e inequality. For $\weightIdi$ and $\weightUi$, the function $g$ is explicit. For $\weightGi$, as $g$ saturates a classical (unweighted) Poincar\'e inequality, we compute it with the function
\Rnot{PoincareOptimal}, using a sequence of $500$ equally spaced nodes. For $\weightSTIMi$, we choose $g$ as the monotonic estimator $\fSTIMi$ of $f_i$ given by the Shape Constrained Additive Model (SCAM) \cite{SCAM} from a sample of size $150$. 
Finally, given $g$, the weights $\weightIdi$, $\weightGi$, $\weightUi$ and $\weightSTIMi$
are approximated with the numerical method described in Section \ref{Numerical computation}, using a sequence of $500$ equally spaced nodes. \medskip 

\noindent \textit{Monte-Carlo estimation and sample size}
\hfill\\
Monte-Carlo integration is employed for computing the variance, weighted DGSM, and \PCE scalar products, using a sample of input vectors and their corresponding outputs. When the derivatives of the outputs are unknown, we estimate them using finite differences. 
We have used a reasonably small sample size of $150$ for the three models considered, whose dimensions are $5$ for the two first ones, and $8$ for the flood model.
Additionally, due to estimation error in Monte-Carlo simulations, we have also used a larger sample of size $10\,000$ when computing the \PCE approximations, in order to highlight the benefits of incorporating weights. Furthermore, we perform $100$ bootstrap replicates for each estimation of the upper bounds and approximations. They are displayed with boxplots to represent confidence intervals.

\subsection{Illustration with toy models}
The two toy models we consider in this part depend on $X=(X_1,\dots,X_5)$, a vector of i.i.d. random variables with common distribution $\mu_i$ being uniform on $(0,1)$. As such, for each $i\in\{1,\dots,d\}$, the weight $\weightUi$ takes the constant value $\weightUi\equiv C_P(\mu_i,1)=1/\pi^2$ and thus their upper bounds and \PCE approximations match with the classical ones. We exclude $\weightUi$ in the presentation of our results below. \medskip 

\noindent 
\textit{A simple polynomial toy model}
\hfill\\
As a first toy model we propose
\begin{equation} \label{eq:toyModel1}
    f(X) = X_1+X_2^2+X_3^3+X_4^4+X_5^5.
\end{equation}
Then for every $i\in \{1,\dots,d\}$, the function $f_i^\tot$ coincides with the main effect $f_i$ and we have 
$$
f_i^\tot (X_i) = f_i (X_i) = X_i ^i - \frac{1}{1+i}.
$$ 
Then the theoretical Sobol index in this case is
\[
S_i^\tot =\frac{\frac{1}{2i+1}-\frac{1}{(1+i)^2}}{\sum_{j=1}^5 \left(\frac{1}{2j+1}-\frac{1}{(1+j)^2}\right)}.\]
Figure \ref{fig: Model1} displays the upper bounds for total Sobol indices. We can see that for each input variable $X_i$, the data-driven weight $\weightSTIMi$ gives the most accurate result. This is expected because our toy function is a model with separated variables and has monotonic main effects. Indeed, from Proposition \ref{prop: weightstability}, the optimal weight is $\weight_{f_i}$, and the upper bound equals the total Sobol index. Furthermore, the estimation of the main effects is here very accurate, as shown in Figure \ref{figure: tendencies F1}. Note additionally that the classical (unweighted) upper bounds provides the worst result.

\begin{figure}[h!t]
  \centering
  \includegraphics[width=.55\linewidth]{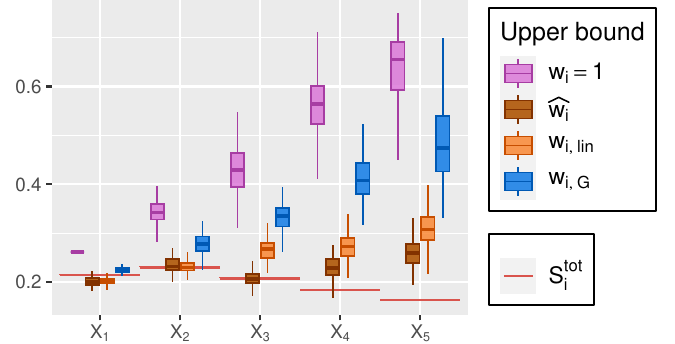}
    \caption{
    Upper bounds on the total Sobol indices 
    for the toy model \eqref{eq:toyModel1}.     Horizontal red bars indicate the true values.
    }
    \label{fig: Model1}
\end{figure}
\begin{figure}[h!t]
    \centering
    \includegraphics[width=.9\linewidth]{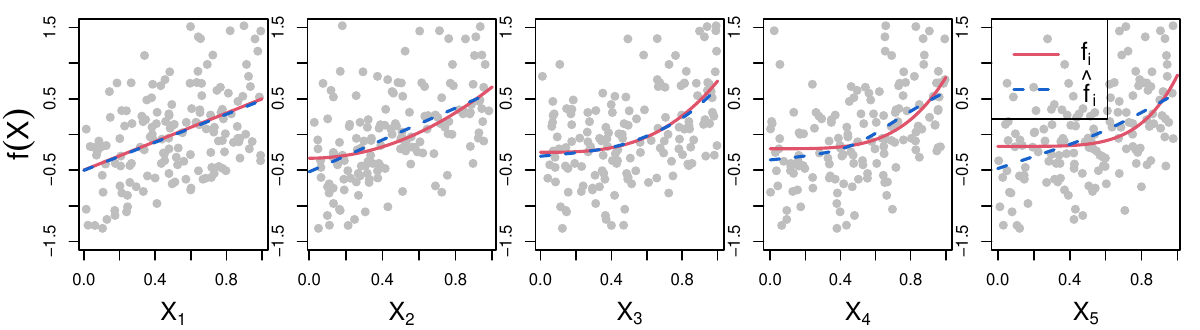}
    \caption{
    Monotonic estimators $\fSTIMi$ of the main effects $f_i$ for the toy model \eqref{eq:toyModel1}. In this figure, the output has been recentered.}
    \label{figure: tendencies F1}
\end{figure}

Next, Figure \ref{fig: Model1 chaos} presents the \PCE approximations of each total Sobol index. The estimations using a large sample size indicate that, for each $X_i$, the weight $\weightGi$ produces the most accurate approximations when comparing with the classical (unweighted) case. Finally, note that the derivative-based approximations exhibit considerably less variance than the derivative-free ones. This is because the
derivatives of the function $f$ have less fluctuations than the function itself.
\begin{figure}[h!t]
    \centering
\begin{subfigure}{.34\textwidth}
  \centering
  \includegraphics[width=.9\linewidth]{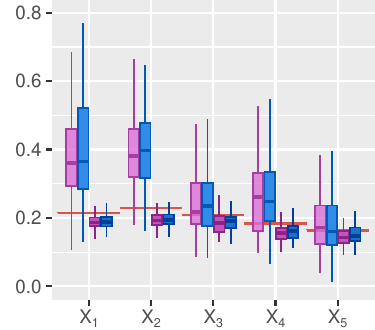}
\end{subfigure}
\begin{subfigure}{.55\textwidth}
  \centering
  \includegraphics[width=1\linewidth]{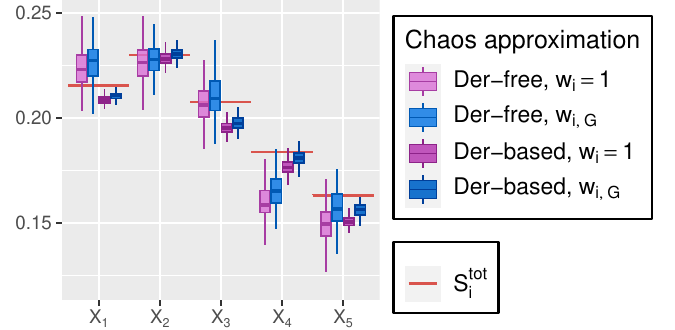}
\end{subfigure}
    \caption{\PCE approximations of the total Sobol indices for the toy model \eqref{eq:toyModel1}, estimated with samples of size $150$ (left) and $10\,000$ (right).
    Horizontal red bars indicate the true values.}
    \label{fig: Model1 chaos}
\end{figure}

\subsubsection*{\textit{A toy model separating variables, with monotonic main effects}}
\hfill \\
Let us propose a second toy model admitting interactions between input variables. Such a model, which might be useful to create toy examples to test the relevance of GSA methodologies, is of the following form: 
\begin{equation}
\label{eq: general model}
    f(X) = \prod_{i=1}^d \left(\frac{h_i(X_i)}{1+a_i} + 1\right),
\end{equation}
where $X=(X_1,\dots,X_d)$ is a vector of independent random variables, $a \in \R^d $ is a constant vector with non-negative coordinates and each $h_i$ is a one-dimensional function such that $\Esp\left[h_i(X_i)\right]=0$. Such formulation is inspired by the $g$-Sobol function, classical in GSA, defined with the particular choice $h_i(x_i)=4\abs{x_i-1/2}-1$. Actually, it is clear that the main effect of each input variable $X_i$ is $f_i(X_i) = h_i(X_i) / (1+a_i)$, so that $a_i$ determines its main influence. Moreover, $a_i$ determines the total influence as well. Indeed, since 
$$f_i^\tot(X) = \frac{h_i(X_i)}{1+a_i} \, \prod_{j\neq i} \left(1+\frac{h_j(X_j)}{1+a_j}\right),
$$ 
we thus obtain after some brief computations, 
$$\Var(f_i^\tot(X)) = \frac{r_i}{(1+a_i)^2} \, \prod_{j\neq i}\left(1+\frac{r_j}{(1+a_j)^2}\right), $$
and
$$\Var(f(X)) = \prod_{j=1}^d \left( 1+ \frac{r_j}{(1+a_j)^2} \right)-1 , 
$$
where $r_i = \Esp \left[ h_i(X_i)^2 \right]$, and one deduces that the total Sobol index $S_i^\tot$ is small as $a_i$ is large. 

We now consider a particular case of 
\eqref{eq: general model}
by choosing $h_i(x_i) = x_i^4 - \frac{1}{5}$ ($i=1, \dots, 5$) and  $a=(1,2,4.5,90,90)$. In particular, each main effect $f_i$ is monotonic.
Figure \ref{fig: Model2} presents the upper bounds. Once again, for each variable $X_i$, the most efficient weight is $\weightSTIMi$. This illustrates the advantage of the data-driven approach with models presenting interactions between variables. 
Note also that for the input variable $X_1$, it provides the unique relevant upper bound (the classical upper bound in the unweighted case $w_i =1$ and the one using $\weight_{i,\text{G}}$ with $i=1$ do not appear in the figure as they are greater than one).
The precision of these upper bounds is explained by the accuracy of each estimator $\fSTIMi$, as we can see in Figure \ref{figure: tendencies F2}.
Concerning the \PCE approximations, the best estimations are once again the ones provided by the weight $\weightGi$ (see Figure \ref{fig: Model2 chaos}).
\begin{figure}[ht]
  \centering
  \includegraphics[width=.52\linewidth]{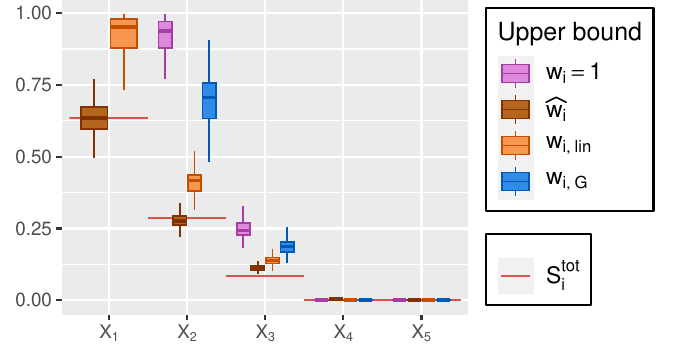}
   \caption{Upper bounds on the total Sobol indices for the toy model \eqref{eq: general model}.
   Horizontal red bars indicate the true values. 
}
    \label{fig: Model2}
\end{figure}
\begin{figure}[ht]
    \centering
    \includegraphics[width=.9\linewidth]{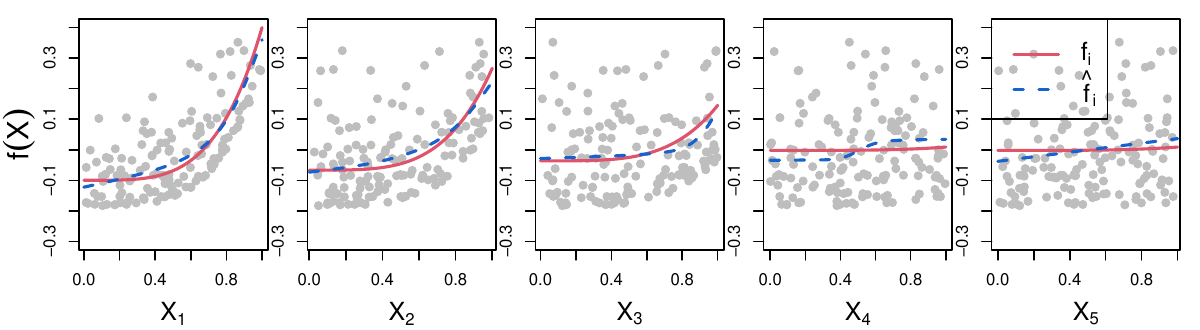}
    \caption{
    Monotonic estimators $\fSTIMi$ of the main effects $f_i$ for the toy model \eqref{eq: general model}. In this figure, the output has been recentered.}
    \label{figure: tendencies F2}
\end{figure}

\begin{figure}[h!t]
    \centering
\begin{subfigure}{.34\textwidth}
  \centering
  \includegraphics[width=.9\linewidth]{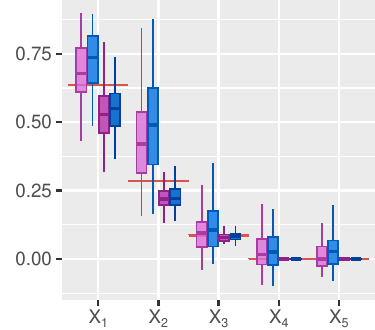}
\end{subfigure}
\begin{subfigure}{.55\textwidth}
  \centering
  \includegraphics[width=1\linewidth]{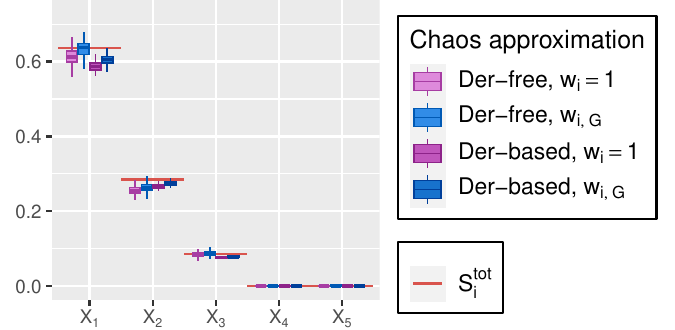}
\end{subfigure}
    \caption{
    \PCE approximations, estimated with samples of sizes $150$ (left) and $10\,000$ (right), of the total Sobol indices for the toy model \eqref{eq: general model}. Horizontal red bars indicate the true values.
    }
    \label{fig: Model2 chaos}
\end{figure}
\subsection{Application to a flood model}
\label{Application to a flood model}
To finish this work, we consider a simplified flood model commonly used to test GSA methodologies. It has been addressed, for instance, in  \cite{iooss2014,Chaos2,poincareintervals,PoincareChaos}.
The outputs of interest are:
\begin{itemize}
    \item the maximal annual overflow (measured in meters)
\begin{equation*}
\label{eq: overflow}
S=Z_v-H_d-C_b+\left(\frac{Q}{B K_s}\sqrt{\frac{L}{Z_m-Z_v}}\right)^{\frac{3}{5}}.
\end{equation*}
\item the annual cost of the dyke maintenance in million of euros
\begin{equation*}
\label{eq: cost}
	C=\mathbbm{1}_{S>0}+\left(0.2+0.8\left(1-e^{-\frac{1000}{S^4}}\right)\right)\mathbbm{1}_{S\leq 0}+\frac{1}{20}\max\left\{ H_d,8\right \}.
\end{equation*}
\end{itemize}
The inputs are supposed to be independent random variables, whose distributions are in accordance with the empirical distributions obtained from measurement campaigns. Details are given in Table \ref{fig: variables flood}.
\begin{table}[H]
\begin{center}
\bgroup
\def\arraystretch{1.2}
\begin{tabular}{llll}
\hline
Input & Meaning & Unit & Probability measure   \\ \hline 
$X_1=Q$ & Max. flow rate & $m^3/s$ & Gumbel $\mathcal{G}(1013, 558)|_{[500, 3000]}$ \\  
$X_2=K_s$ & Strickler coefficient  & --- & Gaussian $\mathcal{N}(30,64)|_{[15,75]}$ \\ 
$X_3=Z_v$ & Downstream level & $m$ & Triangular $\mathcal{T}(49, 50, 51)$\\ 
$X_4=Z_m$ & Upstream level & $m$ & Triangular $\mathcal{T}(54, 55, 56)$ \\
$X_5=H_d$ & Dyke height & $m$ & Uniform $\mathcal{U}(7,9)$\\ 
$X_6=C_b$ & Bank height & $m$ & Triangular $\mathcal{T}(55, 55.5, 56)$ \\
$X_7=L$ & River length & $m$ &  Triangular $\mathcal{T}(4990, 5000, 5010)$ \\ 
$X_8=B$ & River width & $m$ & Triangular $\mathcal{T}(295, 300, 305)$ \\ 
\hline
\end{tabular}
\egroup
\caption{Input variables for the flood model. The notation $|_I$ means that the distribution is truncated on the interval $I$.}
\label{fig: variables flood}
\end{center}
\end{table}
The pdf of the Gumbel distribution $\mathcal{G}(\eta,\beta)$ ($\eta\in \R$, $\beta>0$) and of the triangular distribution $\mathcal{T}(a,c,b)$ ($a<c<b$) are given respectively by:
\begin{equation*}
\label{eq: Gumbel}
    \rho(x)=\frac{1}{\beta}\exp\left(-\frac{x-\eta}{\beta}-\exp\left(-\frac{x-\eta}{\beta}\right)\right),\quad x\in \R,
\end{equation*}
\begin{equation*}
    \rho(x)=\frac{2(x-a)}{(b-a)(c-a)}\mathbbm{1}_{[a,c]}(x)+\frac{2(b-x)}{(b-a)(b-c)}\mathbbm{1}_{(c,b]}(x),\quad x\in \R.
\end{equation*}
For the sake of comparison, we need the total Sobol indices and the main effects. Since closed form expressions can be hardly obtained, we estimate them with a large sample of size $10\,000$ and refer to these estimations as ``true'' values. More precisely, the total Sobol indices are computed with the function {\fontfamily{qcr}\selectfont soboljansen} from the package {\fontfamily{qcr}\selectfont sensitivity} \cite{R_sensitivity} and the main effects are estimated with the function {\fontfamily{qcr}\selectfont loess} (package 
 {\fontfamily{qcr}\selectfont stats}).

We present the results for $S$ and $C$ at the same time, focusing on the variables $Q$, $K_s$, $Z_v$ and $H_d$, which are found to be the most influential. 
Figure \ref{fig: UpperBoth} displays the upper bounds. We see that for each variable $X_i$, the bounds provided by the data-driven weight $\weightSTIMi$ and $\weightIdi$ are sharp and somewhat similar, except for $H_d$ in the output $C$, where the bounds are rather rough. This poor result for $H_d$ may be due to the fact that its main effect has a non-monotonic quadratic shape (see Figure \ref{fig: tendencies both}). Recall that when the main effect is non monotonic, the data-driven weight can be constructed but does not come with a theoretical guarantee. Furthermore, this may cause numerical troubles. For instance, we could not use the SCAM estimator as it produces a flat curve in the region $H_d < 8$ which creates a singularity in the definition of the estimated weight given in \eqref{eq: weight gn}. We have replaced it by an increasing piecewise affine function with knots $H_d=7, 8, 9$.
\begin{figure}[ht]
    \centering
\begin{subfigure}{.35\textwidth}
  \centering
  \includegraphics[width=1\linewidth]{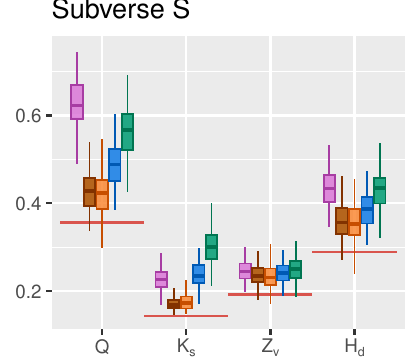}
\end{subfigure}
\begin{subfigure}{.52\textwidth}
  \centering
  \includegraphics[width=1\linewidth]{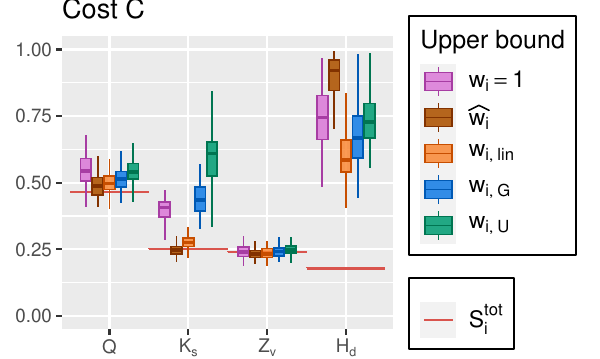}
\end{subfigure}
    \caption{
    Upper bounds of the total Sobol indices for the flood model. Horizontal red bars indicate the true values.
    }
    \label{fig: UpperBoth}
\end{figure}
\begin{figure}[ht]
    \centering
\begin{subfigure}{1\textwidth}
  \centering
  \includegraphics[width=.7\linewidth]{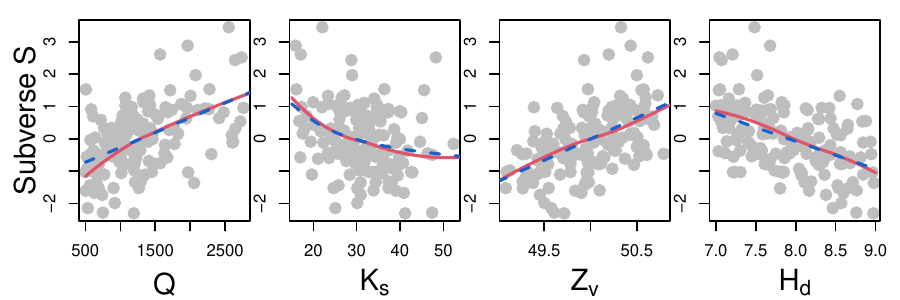}
  \includegraphics[width=.7\linewidth]{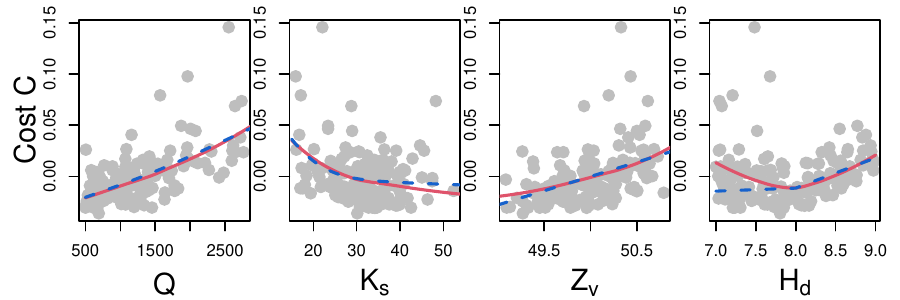}
\end{subfigure}
    \caption{Estimators $\fSTIMi$ (blue lines) of the main effects $f_i$ (red lines) for the recentered outputs of the flood model.}
    \label{fig: tendencies both}
\end{figure}

We present now the \PCE approximations in Figures
\ref{fig: ChaosOutput1} and \ref{fig: ChaosOutput2}.
We obtain relevant and similarly accurate approximations for
each variable. This includes the case of the variable $H_d$ in the output $C$, where
no accurate upper bound was provided. This is remarkable, considering that only few basis functions have been used in the \PCE chaos. Using more basis functions would reduce the slight bias for the output $C$.
\begin{figure}[h!t]
    \centering
\begin{subfigure}{.38\textwidth}
  \centering
  \includegraphics[width=.9\linewidth]{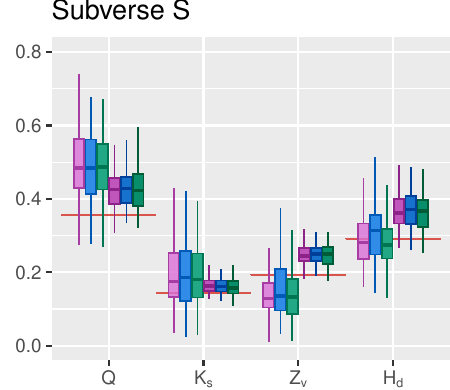}
\end{subfigure}
\begin{subfigure}{.55\textwidth}
  \centering
  \includegraphics[width=1\linewidth]{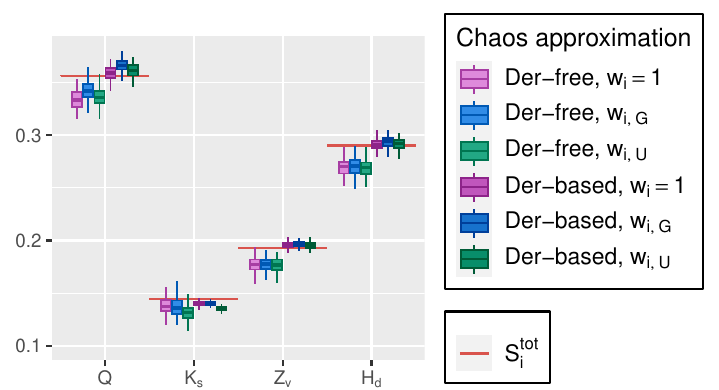}
\end{subfigure}
    \caption{
    \PCE approximations, estimated with samples of sizes $150$ (left) and $10\,000$ (right), of the total Sobol indices for $S$. Horizontal red bars indicate the true values.
}
    \label{fig: ChaosOutput1}
\end{figure}
\begin{figure}[h!t]
    \centering
\begin{subfigure}{.38\textwidth}
  \centering
  \includegraphics[width=.9\linewidth]{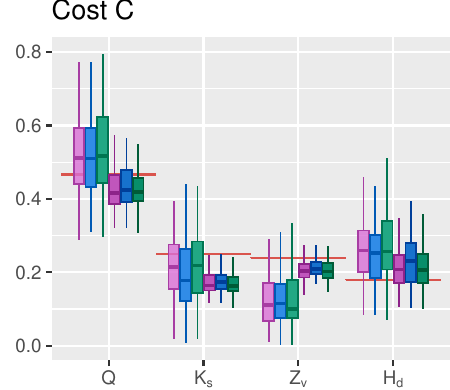}
\end{subfigure}
\begin{subfigure}{.55\textwidth}
  \centering
  \includegraphics[width=1\linewidth]{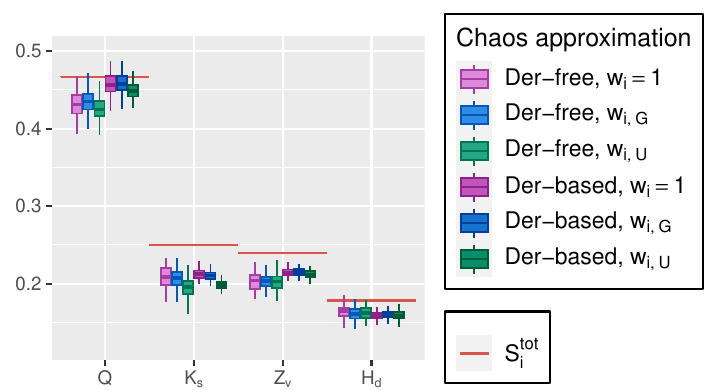}
\end{subfigure}
    \caption{
    \PCE approximations, estimated with samples of sizes $150$ (left) and $10\,000$ (right), of the total Sobol indices for the flood model (cost output). Horizontal red bars indicate the true values.
}
    \label{fig: ChaosOutput2}
\end{figure}

\appendix

\section*{Acknowledgement}
This research was supported by the ANR LabEx CIMI \textit{Global sensitivity analysis and Poincar\'e inequalities} (grant ANR-11-LABX-0040) within the French State Programme ``Investissements d'Avenir''.

\section{}
\label{Computations of weights adapted for linear functions}
This short appendix presents the various computations of the weight $\weightId$ for each example listed in Table \ref{table: weights}. We deal only with the truncated measures, as the weights in the non truncated case are immediately recovered by letting the truncation boundaries tend to infinity. Given some probability measure $\mu\in \probClass$, recall that by Theorem \ref{teo: weight} the weight $\weightId$ is given by 
$$
\weightId (x) = \frac{1}{\rho(x)} \int_{a}^x \left( m - y \right)  \, \rho(y)\,dy, \quad x\in [a,b], 
$$
where $m$ denotes the mean $m=\int_a^b z\,\rho(z)\,dz$ and $\rho$ is the Lebesgue density of the measure $\mu$. 

\addtocontents{toc}{\protect\setcounter{tocdepth}{1}}
\mathsubsection{Uniform distribution}{\mathcal{U}(a,b)}
For all $x\in [a,b]$,
\begin{equation*}
\scaleobj{.95}{
    \weightId(x) = \frac{a+b}{2}\int_a^x dy-\int_a^x y\,dy 
= \frac{1}{2}(x-a)(b-x).
}
\end{equation*}

\mathsubsection{Truncated exponential distribution}{\mathcal{E}(\gamma)|_{[0,h]}}
In this case the mean is
$$
m = \left(\int_0^h  \gamma\, e^{-\gamma z}dz\right)^{-1}\int_0^h \gamma\, z\, e^{-\gamma z} dz = \frac{1}{\gamma} -\frac{h e^{-\gamma h}}{1-e^{-\gamma h}},
$$
and for all $x\in [0,h]$,
\begin{align*}
    \weightId(x) & = e^{\gamma x}\left(\left( \frac{1}{\gamma} -\frac{h\, e^{-\gamma h}}{1-e^{-\gamma h}}\right)\int_0^x e^{-\gamma y}dy-\int_0^x y\, e^{-\gamma y}dy\right)\\
    & = e^{\gamma x}\left(\left(\frac{1}{\gamma} -\frac{h\, e^{-\gamma h}}{1-e^{-\gamma h}} \right)\frac{1}{\gamma}\left(1-e^{-\gamma x}\right)+\frac{1}{\gamma}x\,e^{-\gamma x}-\frac{1}{\gamma^2}\left(1-e^{-\gamma x}\right)\right)\\
    & = \frac{1}{\gamma}\left(x-h \, \frac{e^{\gamma x}-1}{e^{\gamma h}-1}\right).
\end{align*}
\mathsubsection{Truncated normal distribution}{\mathcal{N}(m,\sigma^2)|_{[m-h,m+h]}}
For all $x\in [m-h,m+h]$,
\begin{align*}
    \weightId(x) & = e^{(x-m)^2/2\sigma^2} \int_{m-h}^x (m-y) \, e^{-(y-m)^2/2\sigma^2} dy\\
    & = \sigma^2 \, e^{(x-m)^2/2\sigma^2}\left(e^{-(x-m)^2/ 2\sigma^2} - e^{-h^2/2\sigma^2} \right) \\
    & = \sigma^2 \left(1-e^{((x-m)^2-h^2)/ 2\sigma^2}\right).
\end{align*}
\mathsubsection{Truncated generalized Cauchy distribution}{\mathcal{C} (\beta)|_{[-h,h]}}
We have $m=0$ and for all $x\in [-h,h]$,
    \begin{align*}
    \weightId(x)&=-(1+x^2)^{\beta}\int_{-h}^x y\,(1+y^2)^{-\beta}\,dy\\
    &=\left\{\begin{array}{ll}\displaystyle
        \frac{1}{2(\beta-1)}\left(\left(1+x^2\right)-(1+x^2)^{\beta}(1+h^2)^{-\beta+1}\right)
         & \mbox{if }\beta\neq 1; \\
        \displaystyle -\frac{1}{2}(1+x^2)\log\left(\frac{1+x^2}{1+h^2}\right)
        & \mbox{if }\beta=1.
    \end{array}\right.
\end{align*}
\mathsubsection{Truncated Pareto distribution}{\mathcal{P} ar (z,\alpha)|_{[z,z+h]}}
$\bullet$ Case $\alpha\neq 1$: the mean is
\begin{align*}
    m = \frac{\alpha}{z^{-\alpha}-(z+h)^{-\alpha}}\int_z^{z+h} y^{-\alpha}dy = \frac{\alpha}{\alpha-1} \, \frac{z^{1-\alpha}-(z+h)^{1-\alpha}}{z^{-\alpha}-(z+h)^{-\alpha}}, 
\end{align*}
and for all $x\in [z,z+h]$,
\begin{align*}
    \weightId(x) &= x^{\alpha+1}\left(m\int_z^x y^{-(\alpha+1)}dy-\int_z^x y^{-\alpha}dy\right)\\
    & = x^{\alpha+1}\left(\frac{m}{\alpha}\left(z^{-\alpha}-x^{-\alpha}\right)-\frac{1}{\alpha-1}(z^{1-\alpha}-x^{1-\alpha})\right)\\
    & = \frac{x^{\alpha+1}}{\alpha-1} \left(\frac{z^{1-\alpha}-(z+h)^{1-\alpha}}{z^{-\alpha}-(z+h)^{-\alpha}}\left(z^{-\alpha}-x^{-\alpha}\right)-\left(z^{1-\alpha}-x^{1-\alpha}\right)\right).
\end{align*}
\smallskip 

$\bullet$ Case $\alpha=1$: we have $m=(\log(z+h)-\log(z))/(z^{-1}-(z+h)^{-1})$ and for all $x\in [z,z+h]$, 
 \begin{align*}
     \weightId(x)&=x^2\left(m\int_z^x y^{-2}dy-\int_z^x y^{-1}dy\right)\\
     &=x^2\left((\log(z+h)-\log(z)) \, \frac{z^{-1}-x^{-1}}{z^{-1}-(z+h)^{-1}}-(\log(x)-\log(z))\right).
 \end{align*}

\bigskip
\noindent
(D. Heredia, corresponding author) UMR CNRS 5219,  \textsc{Institut de Math\'ematiques de
Toulouse, Universit\'e de Toulouse, France}\par
\textit{E-mail address:} \href{mailto:dheredia@insa-toulouse.fr}{mailto:dheredia(at)insa-toulouse.fr}\par
\textit{URL:} \url{https://davidherediag.wordpress.com/}\\ \\
(A. Joulin) UMR CNRS 5219, \textsc{Institut de Math\'ematiques de
Toulouse, Universit\'e de Toulouse, France}\par
\textit{E-mail address:} \href{mailto:ajoulin@insa-toulouse.fr}{mailto:ajoulin(at)insa-toulouse.fr}\par
\textit{URL:} \url{https://perso.math.univ-toulouse.fr/joulin/}
\\ \\
(O. Roustant) UMR CNRS 5219, \textsc{Institut de Math\'ematiques de
Toulouse, Universit\'e de Toulouse; INSA Toulouse, France.}\par
\textit{E-mail address:} \href{mailto:roustant@insa-toulouse.fr}{mailto:roustant(at)insa-toulouse.fr}\par
\textit{URL:} \url{https://olivier-roustant.fr/}

\end{document}